\documentclass[12pt,a4paper]{article}
\usepackage[margin=1.2in]{geometry}
\usepackage{pstricks}
\usepackage{graphicx,amssymb,color,xcolor}
\usepackage{epstopdf}
\usepackage{amsthm}
\usepackage{thmtools}
\usepackage{bm}
\usepackage{authblk}
\usepackage{tikz}
\usepackage{tikz-3dplot}
\usepackage{comment}
\usepackage{multicol}

\usepackage{xpatch}
\makeatletter
\xpatchcmd{\@thm}{\thm@headpunct{.}}{\thm@headpunct{}}{}{}
\makeatother

\newcommand\bsquare{$\blacksquare$}

\newtheorem{theorem}{Theorem}[section]
\newtheorem{lem}[theorem]{Lemma}
\newtheorem{prop}[theorem]{Proposition}

\newtheorem{defn}[theorem]{Definition}

\declaretheoremstyle[
spaceabove=6pt, spacebelow=6pt,
headfont=\normalfont\bfseries,
notefont=\m{\rm d}series, notebraces={(}{)},
bodyfont=\normalfont,
headpunct={},
postheadspace=1em,
qed=\bsquare
]{examplestyle}

\theoremstyle{definition}
\newtheorem{rem}[theorem]{Remark}

\usepackage{amsmath,amssymb}
\usepackage{amsfonts}

\newcommand{\eukao}{\mathrm{E}^{\kappa_1}}
\newcommand{\eukat}{\mathrm{E}^{\kappa_2}}

\newcommand{\iup}{\iota}

\newcommand{\Cay}{\Phi}

\newcommand{\RM}{Rotation Minimizing\ }
\newcommand{\FS}{Frenet--Serret\ }

\newcommand{\hide}[1]{}
\renewcommand\thmcontinues[1]{cont.}

\definecolor{sapgreen}{rgb}{0.31, 0.49, 0.16}

\title{On the use of the \RM  Frame for Variational Systems with Euclidean Symmetry}
\author[]{ E. L. Mansfield}   
\author[]{A. Rojo-Echebur\'{u}a}
\affil[]{SMSAS, University of Kent, Canterbury, CT2 7FS, UK}

\date{}

\begin{document}
\maketitle

\begin{abstract} 
We study variational systems for space curves, for which the Lagrangian or action principle has a Euclidean symmetry, using the \RM  frame, also known as the 
Normal, Parallel or Bishop frame (see  \cite{Bishop}, \cite{WangJoe}).  

  Such systems have previously been studied using the Frenet--Serret frame. However, the \RM frame has many advantages,
   and can be used to study a wider class of examples.

We achieve our results by extending the powerful symbolic invariant calculus for Lie group based
moving frames, to the \RM frame case. To date, the invariant calculus has been developed for frames defined 
by algebraic equations. By contrast, the \RM frame is defined by a differential equation.

In this paper, we derive the recurrence formulae for
the symbolic invariant differentiation of  the symbolic invariants. We then derive
  the syzygy operator needed to obtain Noether's conservation laws as well as the Euler--Lagrange equations directly in terms of the invariants,
    for variational problems
with a Euclidean symmetry. We show how to
 use the six Noether laws to ease the integration problem for the minimizing curve, once the Euler--Lagrange equations
have been solved for the generating differential invariants. 
Our applications include variational problems used in the study of strands of proteins, nucleid acids and polymers.
\end{abstract}


\section{Introduction}

The study of variational problems with Euclidean symmetry is an old problem, indeed, Euler's 1744 study of elastic beams 
is such a case. However, methods to analyse such problems efficiently and effectively, are still of interest. 

In this paper, we consider variational problems for curves in 3-space for which the Lagrangian is invariant under 
the special Euclidean group $SE(3)=SO(3)\ltimes \mathbb{R}^3$ acting linearly in the standard way, that is,
\begin{equation}\label{SE3act}
\begin{pmatrix} x\\y\\z\end{pmatrix} \mapsto R\begin{pmatrix} x\\y\\z\end{pmatrix} + \begin{pmatrix} a\\b\\c\end{pmatrix},\qquad R\in SO(3).
\end{equation}
The Euler--Lagrange equations satisfied by the extremising curves have $SE(3)$ as a Lie symmetry group, and can be therefore be written
in terms of the differential invariants of the action, and their derivatives with respect to arc-length.  Further, the six dimensional space of Noether's laws
are key to analysing the space of extremals. 

 To date, the \FS frame has been used to analyse Euclidean invariant variational problems, and this requires that the Lagrangian can be written in terms of the Euclidean curvature and torsion.  Because the \FS frame can be derived using \textit{algebraic} equations (at each point) on the relevant jet bundle, 
 the powerful symbolic calculus of invariants can be used, to obtain not only the Euler--Lagrange equations directly in terms of the curvature and torsion, but the full set of Noether's laws can also be written down directly using both the invariants and the frame \cite{GonMan2}.
 
Let us denote the space curve as $s\mapsto P(s)\in\mathbb{R}^3$, where $s$ is arc-length, and the tangent vector to this curve by $P'$, so that
$'={\rm d}/{\rm d}s$. By the definition of arc-length, $|P'|^2=P'\cdot P'=1$. Then provided $P''\ne 0$, the left Frenet--Serret frame is given 
by 
\begin{equation}\label{FSdef}
\sigma_{FS}^{\ell}=\begin{pmatrix} P'(s) & \frac{P''(s)}{|P''(s)|} & \frac{P'(s) \times P''(s)}{|P''(s)|}\end{pmatrix} \in SO(3).
\end{equation}

 From a computational point of view, the \FS frame is convenient as it can be computed straightforwardly at arbitrary points along the curve. However, it is undefined wherever the curvature is degenerate, such as at inflection points or along straight sections of the curve. 
 The left \FS frame is \textit{left equivariant}, that is, if at any point $z=P(s)$ on the curve, since 
 $R\in SO(3)$ acts linearly in the standard way on the tangent space $T_z\mathbb{R}^3$, then it is readily seen that
 \[ \sigma_{FS}^{\ell} \mapsto \begin{pmatrix} R P'(s) & \frac{R P''(s)}{| R P''(s)|} & \frac{R P'(s) \times R P''(s)}{| R P''(s)|}\end{pmatrix}=R \sigma^{\ell}_{FS}.\]
  The Euclidean curvature $\kappa$ and the torsion $\tau$ at the point $P(s)$ are then the nonzero components of the invariant so-called curvature matrix, specifically, 
  \begin{equation}\label{FScurvMx} \left(\sigma^{\ell}_{FS}\right)^{-1} \left(\sigma^{\ell}_{FS}\right)' = \begin{pmatrix}0 & -\kappa &0\\ \kappa &0&-\tau\\ 0&\tau&0\end{pmatrix}.\end{equation}
 
 In contrast to this frame, \textit{relatively parallel} frames were described by \cite{Bishop} who detailed what is now known variously as the Normal, Parallel, Bishop or \RM frame. 
 The \RM frame has many advantages over the \FS frame. First of all, unlike the \FS frame, the \RM frame is
 defined at all points of a smooth curve. 
The \RM frame may be used to study a larger class of variational problems, because
while the generating invariants for the symbolic invariant calculus given by the \FS frame, curvature and torsion, are of order
2 and 3 respectively, those given by the \RM  frame are both of order only 2. Finally, 
 the \RM frame, its computation, approximation and its applications, have been extensively used and studied in the Computer Aided Design literature,
 \cite{BR,F1,F2,F3,G,H,K,PFL,PW,SW,WJZL}. 
 One reason is that the sweep surfaces they generate are, in general, superior, \cite{WangJoe}; as illustrated in Figure \ref{sweep1}, sweep surfaces
generated from the \FS frame can exhibit strong twisting at inflection points.

Bishop, \cite{Bishop}, defines a normal vector field along a curve to be \textit{relatively parallel} if its derivative is proportional to the tangent vector. 
The equation used in the Computer Aided Design literature for the relatively parallel normal 
vector $V=V(s)$ to the curve $s\mapsto P(s)$ is \cite{WangJoe},
\begin{equation}\label{Vdef} V' = -(P''\cdot V) P'.\end{equation}
The function of proportionality between $V'$ and $P'$ is chosen to guarantee that, without loss of generality,
we may suppose that $|V|\equiv 1$ and $P'\cdot V\equiv 0$, see Proposition \ref{propVprops}.
Then the left \RM frame is
\begin{equation}\label{RMdef}
\sigma_{RM}^{\ell} = \begin{pmatrix} P' & V & P'\times V\end{pmatrix}.
\end{equation}
We  have that $\sigma_{RM}^{\ell}$ is left equivariant and, as shown by Bishop, the invariant curvature matrix $\left(\sigma^{\ell}_{RM}\right)^{-1}\left(\sigma^{\ell}_{RM}\right)'$ takes the form
\begin{equation}\label{RMcurvMx} \left(\sigma^{\ell}_{RM}\right)^{-1}\left(\sigma^{\ell}_{RM}\right)'= \begin{pmatrix} 0 & -\kappa_1 & -\kappa_2\\\kappa_1 & 0 &0\\ \kappa_2 &0&0\end{pmatrix},
\end{equation}
that is, where the $(2,3)$-component is guaranteed to be zero.

\begin{figure}
  \caption{Given a curve in  space, we compare the \FS frame with the \RM frame along it. 
  \label{sweep1}}
  \centering
  
\[
\begin{array}{cc}
\includegraphics[scale=0.32]{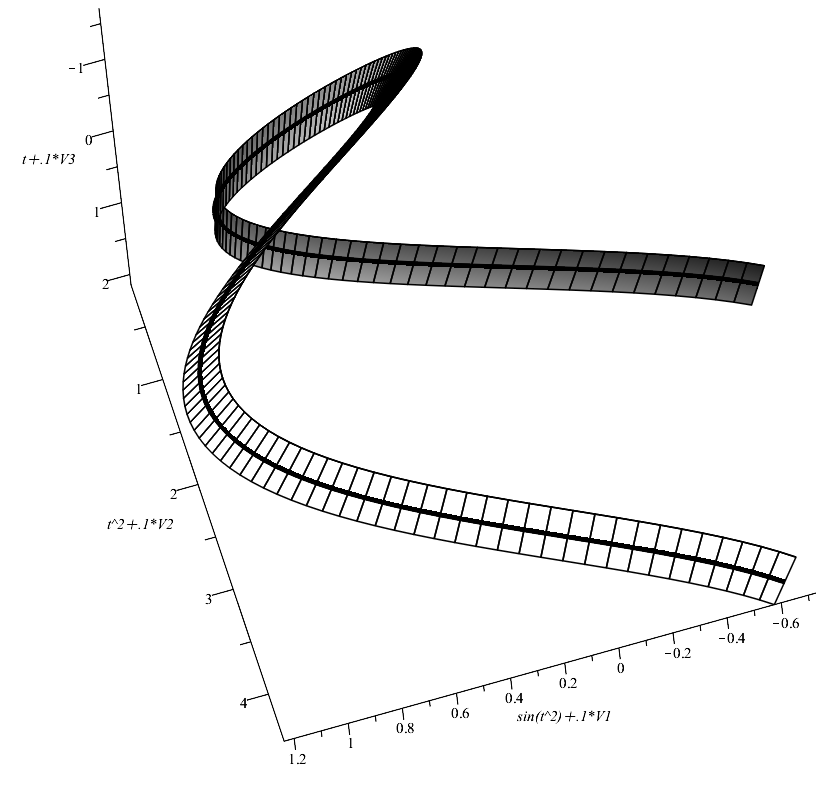} &
\includegraphics[scale=0.3]{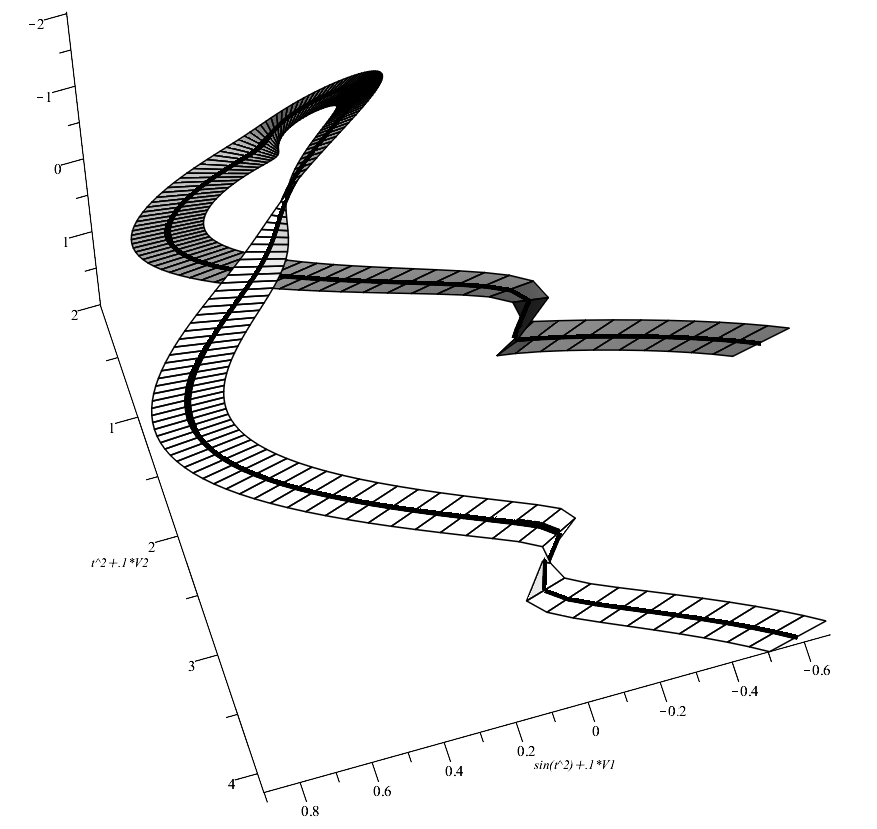}\\[9pt]
\text{Surface sweeping given by} \ V & \text{Surface sweeping given by} \ P''\\
\text{using the \RM frame} & \text{using the \FS frame}
\end{array}
\]
  
\end{figure}

Since both the \RM and the \FS frames share the same first column, we have for some angle $\theta=\theta(s)$, (see Figure \ref{gaugeFSRM}), 
\begin{equation}\label{sigmarm}
\sigma^{\ell}_{RM} = \sigma^{\ell}_{FS} \left( \begin{array}{ccc} 1 & 0 & 0\\ 0 & \cos{\theta} & \sin{\theta} \\ 0 & -\sin{\theta} & \cos{\theta} \end{array}\right).
\end{equation}

\begin{figure}
  \caption{\label{gaugeFSRM}Diagram of a \RM frame and a \FS frame of a curve $P(s)$ in $\mathbb{R}^3$. Note that $P'(s)$ is common in both frames.}
  \centering
    \includegraphics[width=0.8\textwidth]{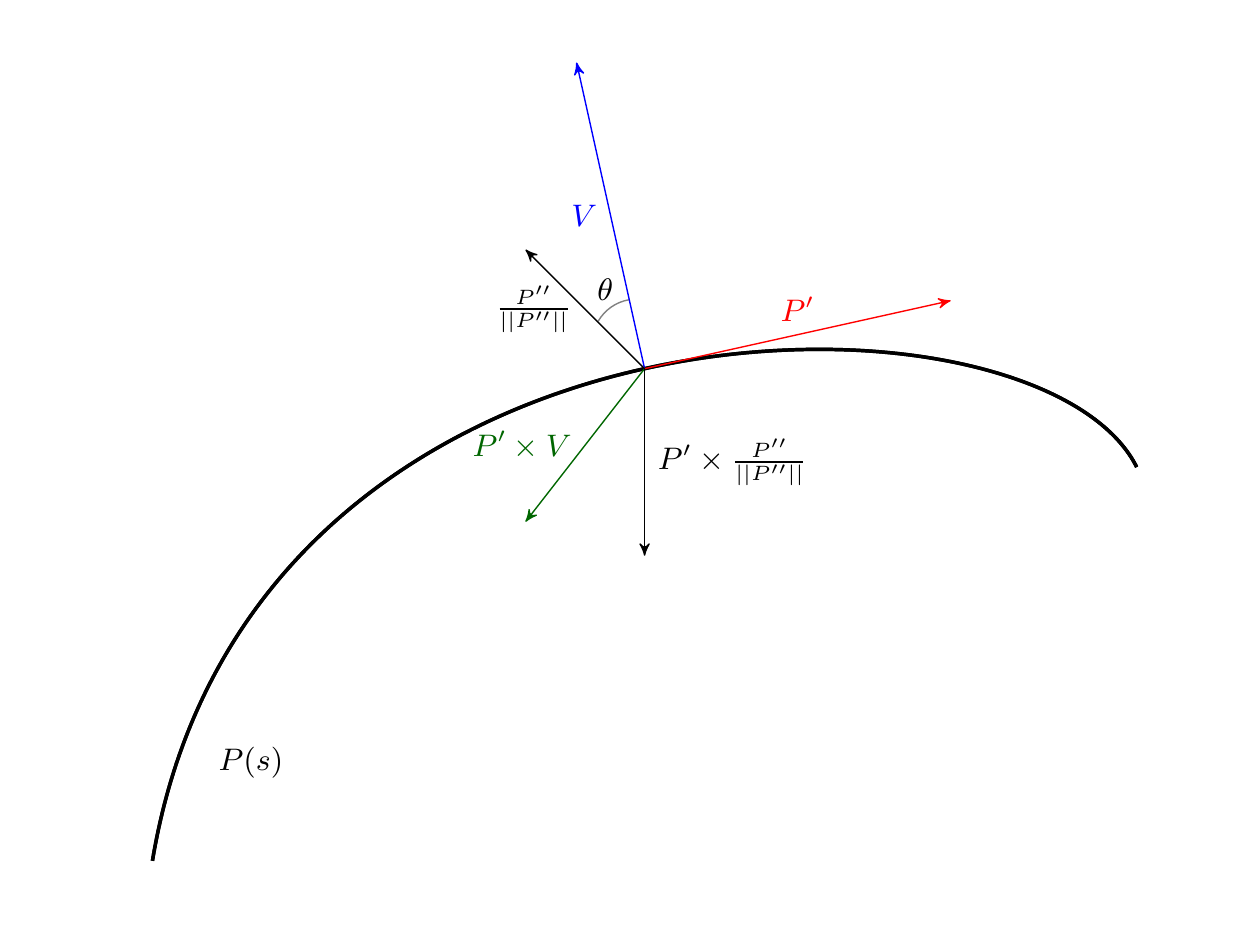}
\end{figure}

Calculating $\left(\sigma^{\ell}_{RM}\right)^{-1} \left(\sigma^{\ell}_{RM}\right)'$, using \eqref{FScurvMx}, and \eqref{sigmarm},   and comparing the result to \eqref{RMcurvMx} leads to
the well known relations,
\begin{equation}\label{k1k2} \kappa_1=\kappa\cos\theta,\qquad \kappa_2 = \kappa\sin\theta, \qquad \theta_s = \tau.\end{equation}

Treating the \RM frame as a gauge transformation of the \FS frame, together with

\[
 \theta(t) - \theta_0 = \int^t_{t_0} \tau(t) |P'(t)| \  {\rm d}t
\]
has been proven to lack numerical robustness for a general space curve, (see \cite{G}).
This makes the use of the \RM frame defined in terms of the normal vector $V$, as in \eqref{RMdef}, 
to be a better choice in the application literature, and is our choice here. 

The theory and applications of Lie group based moving frames are now well established, and provides an invariant calculus to study differential systems that are either invariant or equivariant under the action of a Lie group, (see the graduate text \cite{Mansfield:2010aa} and references therein).
Beginning with \cite{FO2,FO1} who found the recurrence formulae for the invariant differentiation of  a computable set of generating invariants,
given algebraic equations for the frame, 
there is now a rigorous and constructive symbolic invariant calculus,
\cite{hubertAA,hubertAD,hubertAC} and \cite{hubertB,hubertA}. This calculus has been applied to study problems
in the calculus of variations where the Lagrangian has a Lie group symmetry \cite{kogan,GonMan, GonMan2, GonMan3}.
The case of Euclidean invariance using the \FS frame was studied by \cite{GonMan2}. 

The formulae for the recurrence relations in the symbolic invariant calculus require the equations defining the frame to be algebraic at each point in the domain of the frame, and indeed,
the equations defining the \FS frame,
despite involving the components of $P(s)$, $P'(s)$ and $P''(s)$, are algebraic at each point of the relevant jet bundle.
However, the recurrence formulae for the invariant derivatives defined using the \RM frame need to be derived in another way, because the equations defining the frame are not algebraic
in the jet variables. Indeed, considering \eqref{k1k2}, it would seem that the \RM frame is defined by a relation on the
invariants, $\tau$ and $\theta_s$, or, a differential equation on an extended space, one which includes either $\theta$, or $V$.

Our approach  is to extend the manifold on which the group acts, to include the vector $V$ and its derivatives, in such a way that
the differential equation defining $V$ is a simple constraint for our variational problem. Because the group acts linearly on $P'$, $V$ and their derivatives,
it turns out to be straightforward to write down a set of generating invariants, the recurrence formulae for their invariant differentiation and their differential
syzygies. With these to hand, the methods used by \cite{GonMan2} can be adapted to obtain  Euler--Lagrange equations directly in terms of the invariants
and to write down the six Noether conservation laws.

In \S \ref{MFsection}, we introduce the  notions of a Lie-group based moving frame, the frame-defined invariants and the curvature matrices.
The symbolic invariantized form of the curvature matrices for the \RM frame are found, and  we derive the recurrence formulae for the symbolic differential invariants and the syzygy operator we will need in the sequel. In \S \ref{ELsection}, we obtain the Euler--Lagrange equations
and Noether's laws for a Lagrangian with a Euclidean symmetry, using the results of \S \ref{MFsection}. In \S \ref{IPsection}, the use of Noether's laws to ease the integration problem is carried out. In \S \ref{EXsection}, some examples and applications are presented. 
The final section, \S \ref{CCsection} is devoted to our conclusions and the implication of our results to extending the range of applicability of the symbolic invariant calculus to frames not defined algebraically.

\section{Moving frames}\label{MFsection}

\begin{defn}\label{leftact}
Given a Lie group $G$ and a manifold $M$, a \emph{left} action of $G$ on $M$ is defined to be a smooth map
\[G\times M\rightarrow M, \qquad (g, z) \mapsto g\cdot z \]
such that
\[ g\cdot (h\cdot z) = (gh)\cdot z.\]
\end{defn}

\begin{defn}[Moving Frame]
Given a smooth left Lie group action $G\times M\rightarrow M$, a moving frame on the domain $\mathcal{U}\subset M$ is an equivariant map $\rho: \mathcal{U}\rightarrow G$, that is
\[\begin{array}{ll}\rho(g\cdot z) = g\rho(z) &\qquad  \mbox{\emph{left equivariance}}\\
\rho(g\cdot z)=\rho(z)g^{-1} &\qquad \mbox{\emph{right equivariance}}\end{array}\]
\end{defn}
The frame is called left or right accordingly.  Given a left frame, its (group) inverse is a right frame, and vice versa. In practice, the ease of calculation can differ considerably depending on the choice of parity.

Moving frames exist when the action is free and regular on its domain
$\mathcal{U}\subset M$. This means, the orbits foliate $\mathcal{U}$, and for any cross section $\mathcal{K}\subset M$ which is transverse to the orbits $\mathcal{O}(z)$, the set $\mathcal{K}\cap \mathcal{O}(z)$ has just one element, the projection of $z$ onto $\mathcal{K}$, (see \cite{Mansfield:2010aa} for full details).

The standard method to calculate a moving frame for the group action on a neighbourhood $\mathcal{U}\subset M$ of $z$ is as follows.
Using a  cross-section $\mathcal{K}$,  given by a system of equations $\psi_r(z)=0$, for $r=1,\ldots,R,$ where $R$ is the dimension of the group $G$,  one then solves the so-called normalization equations,
\begin{equation}\label{frameEqA}
\psi_r(g\cdot z) = 0, \qquad r=1,\ldots, R,
\end{equation}
for $g$ as a function of $z$.
The solution is the group element $g=\rho(z)$ that maps $z$ to its projection on $\mathcal{K}$.
The conditions on the action above are those for the Implicit Function Theorem to hold (see \cite{Hirsch}), so the solution $\rho$ is unique. A consequence of uniqueness is that
$$
\rho(g\cdot z)=\rho(z) g^{-1},
$$
that is, the frame is  \textit{right equivariant}, as both $\rho(g\cdot z) $ and $ \rho(z) g^{-1}$ solve the equation $\psi_r\left(\rho(g\cdot z)\cdot \left(g\cdot z\right)\right)=0$.

The equivariance of the frame enables one to obtain invariants of the group action. 
\begin{lem}[Normalized Invariants]
Given a left or right action $G\times M \rightarrow M$ and a \textit{right} frame $\rho$, then
 $\iup(z)=\rho(z)\cdot z$, for $z$ in the domain of the frame $\rho$, is invariant under the group action.
\end{lem}

\begin{defn}\label{normInvsnames} The normalized, or frame-defined, invariants are the components of $\iup(z)$.
\end{defn}

When the frame is not known explicitly,  the normalized invariants are said to be known symbolically. The power of the symbolic invariant calculus
derives from the fact that these symbolic invariants can be used effectively in a wide range of calculations.
We now state the Replacement Rule, from which it follows that the normalized invariants provide a set of generators for the algebra of invariants.

\begin{theorem}[Replacement Rule]\label{reprule}
Given a right moving frame on $M$ for the action $G\times M\rightarrow M$,
and an invariant $F(z)$ of this  action, then $F(z)=F(\iup(z))$.
\end{theorem}

\begin{defn}[Invariantization]\label{invOpDef} Given a right moving frame $\rho$, the map $z\mapsto \iup(z)=\rho(z)\cdot z$ is called the \emph{invariantization} of the point $z$, and the map
$F(z)\mapsto F(\iup(z))$, 
is called the \emph{invariantization} of $F$.\end{defn}

In this paper, we will consider derivatives with respect to arc-length $s$ of our curve $s\mapsto P(s)$, where we note that arc-length is a Euclidean invariant, and we will also consider the 
evolution of this curve with respect to a `time' parameter $t$, which we declare to be invariant under our $SE(3)$ action.   In general, if the independent variables of our curves and surfaces, with respect to which we differentiate, are \textit{all} invariant, we may make the following definition.

\begin{defn}[Curvature matrices]
The curvature matrix with respect to the independent variable $x$ is defined to be, for a right frame $\rho$,
\[
Q^x=\left(\frac{{\rm d}}{{\rm d} x} \rho\right) \rho^{-1}.
\]
\end{defn}
The non-constant components of the curvature matrices are differential invariants of the action. These are
referred to as the \textit{curvature invariants} or the Maurer--Cartan invariants.

\subsection{The extended right \RM frame}

Since the symbolic invariant calculus is standardly carried out for a right frame, we consider
a right \RM frame, $\rho_{RM}$, which we need for our application to include the translation component of the 
Special Euclidean group $SE(3)$. We consider 
the Lie group $SE(3)$ to act on an enlarged  manifold (jet bundle) having local coordinates to be the components of
\[ P, P', P'', \dots , P^{(n)}=\frac{{\rm d}^n}{{\rm d}s^n} P, \dots, V, V', V'',\dots V^{(n)}=\frac{{\rm d}^n}{{\rm d}s^n} V, \dots\]
where the left action is, for $g=(R, \mathbf{a})\in SE(3)=SO(3)\ltimes \mathbb{R}^3$,
\[ P\mapsto R P + \mathbf{a},\qquad  P^{(n)}\mapsto R P^{(n)}, n>0, \qquad V^{(n)}\mapsto R V^{(n)}, n\ge 0.\]
In the standard representation of $SE(3)$ in $GL(4,\mathbb{R})$,
\[ g=(R, \mathbf{a})\mapsto \begin{pmatrix} R & \mathbf{a}\\ 0 & 1\end{pmatrix},\]
our \textit{extended right \RM frame} for this action is defined to be, 
\begin{equation}\label{rhoRMdef} \rho_{RM}=\begin{pmatrix} \sigma_{RM} & -\sigma_{RM} P \\ 0 & 1\end{pmatrix}\end{equation}
where 
\begin{equation}\label{defRMright}
\sigma_{RM} = \left(\sigma_{RM}^{\ell}\right)^T\in SO(3).
\end{equation}
The curvature matrix is, by direct calculation and noting that $\sigma_{RM}P'=\begin{pmatrix} 1&0&0\end{pmatrix}^T$, 
\begin{equation}\label{ExCurvMx} Q^s=\rho_{RM}'\rho_{RM}^{-1} = \begin{pmatrix} \sigma_{RM}'\sigma_{RM}^{-1} & \begin{array}{r} -1\\ 0\\ 0\end{array}\\ 0 & \phantom{-}0\end{pmatrix}.\end{equation}

To obtain the complete set of normalized invariants and the (reduced) curvature matrix $\sigma_{RM}'\sigma_{RM}^{-1}$, 
we first consider solutions of the defining equation for $V$.

\begin{prop}\label{propVprops}
Given a curve $s\mapsto P(s)\in \mathbb{R}^3$ such that $P'\cdot P'=|P'|^2=1$, and suppose that $V=V(s)$ satisfies equation \eqref{Vdef}, 
which for convenience we give again here,
\begin{equation}\label{NVdef}
V' =  -(P''\cdot V) P'
\end{equation}
together with the initial conditions $V(s_0)=1$, $V(s_0)\cdot P'(s_0)=0$. Then
\begin{enumerate}
\item $V\cdot P' \equiv 0$
\item $ V\cdot V\equiv 1$
\item For any constant $\psi\in\mathbb{R}$, 
\[ W=\cos\psi\, V + \sin\psi\, P'\times V\]
also solves equation \eqref{NVdef} with $|W|\equiv1$ and $W\cdot P'\equiv 0$
\end{enumerate}
\end{prop}

\begin{proof}
\begin{enumerate}
\item By direct calculation,
the scalar product $V\cdot P'$ is constant with respect to $s$. The result follows from the assumption on the initial data.
\item Equation \eqref{NVdef} implies $V'\cdot V = -(P''\cdot V)(P'\cdot V)=0$  by 1. above. Hence $V\cdot V$ is constant
with respect to $s$. The result follows from the assumption on the initial data.
\item Since \eqref{NVdef} is linear, it suffices to prove that  $W=P'\times V$ also solves Equation \eqref{NVdef}.
We have  by the orthogonality of both $V$ and $P''$ to $V$ that $V= b(s) P'' + c(s) P'\times P''$ for some coefficients $b(s)$, $c(s)$.
Then $P'\times V= b(s) P'\times P'' - c(s)P''$ and
\[\begin{array}{rcl}
(P'\times V)' &=& P''\times V + P'\times V'\\
&=& P''\times V\\ &=& c(s)(P''\cdot P'')  P'.
\end{array}\]
But $P''\cdot (P'\times V)= -c(s) P''\cdot P''$ and hence
\[ W' = -(P'' \cdot W)\cdot  P'\]
as required.
\end{enumerate}
\end{proof}

The proposition shows that if $V$ solves \eqref{NVdef} and for some $s_0$, $V(s_0)$ has unit length and is orthogonal to $P(s_0)$,
then $\sigma_{RM}\in SO(3)$ for all $s$, and this we now assume. 
In the applications, it is necessary to ensure the initial data for $V$ holds when integrating for the frame. The proposition shows further that in fact there is a one-parameter family
of \RM frames, determined by the initial data for $V$.

Let $\mathfrak{so}(3)$ denote the set of $3\times 3$ skew-symmetric matrices, the Lie algebra of $SO(3)$.
We have by direct calculation that
\begin{equation}\label{CurvMat} \sigma_{RM}'\sigma_{RM}^{-1}=\left(\begin{array}{ccc}
0 & P''\cdot V & P''\cdot (P'\times V) \\
-P''\cdot V  & 0 & 0\\
-P''\cdot (P'\times V) & 0 &0\end{array}\right)\in \mathfrak{so}(3)\end{equation}

We now write down the symbolic normalized invariants, and obtain $\sigma_{RM}'\sigma_{RM}^{-1}$ in terms of them. 
We denote the components of $P(s)$ as $P(s)=(X(s), Y(s), Z(s))$ and that
of the $n$-th derivative with respect to $s$ as
$P^{(n)}=(X^{(n)}, Y^{(n)}, Z^{(n)})$.
By construction,
\[ \rho_{RM}\cdot P=0\]  
and by definition of the action,
\[\rho_{RM}\cdot P^{(n)}=\sigma_{RM}P^{(n)},\quad n>0.\]
We now recall the standard symbolic names of these normalized invariants (see Definition \ref{normInvsnames}), as
 \begin{equation}\label{DefIXn} \sigma_{RM} P^{(n)}=( \iup(X^{(n)}), \iup(Y^{(n)}), \iup(Z^{(n)}) )^T.\end{equation}
 Since 
 \[
 ((\iup(X'), \iup(Y'), \iup(Z')){^{T}}= \sigma_{RM} P' = (P'\cdot P', V\cdot P',  (P'\times V)\cdot P'){^{T}}= (1, 0, 0)^T,
 \]
  we  make the following definition.
  \begin{defn}[arc-length constraint]\label{arclengCon}
The equation $ \iup(X')=1$ is denoted as the \emph{arc-length constraint}.
\end{defn}

 Differentiating (\ref{DefIXn}) with respect to $s$, yields
\begin{equation}\label{RecDiff}
 \frac{{\rm d}}{{\rm d} s} \begin{pmatrix} \iup(X^{(n)})\\ \iup(Y^{(n)})\\ \iup(Z^{(n)}) \end{pmatrix} 
 = \frac{{\rm d}}{{\rm d} s}({\sigma_{RM}})\sigma_{RM}^{-1} \begin{pmatrix} \iup(X^{(n)})\\ \iup(Y^{(n)})\\ \iup(Z^{(n)}) \end{pmatrix} 
  + \begin{pmatrix} \iup(X^{(n+1)})\\ \iup(Y^{(n+1)})\\ \iup(Z^{(n+1)}) \end{pmatrix}.
 \end{equation} 
Setting $n=1$ 
and recalling $ \sigma_{RM} P' = (1, 0, 0)^T$, we have  from (\ref{CurvMat}) and (\ref{RecDiff}) that 
 \[ \left(\begin{array}{c} 0\\ 0\\0\end{array}\right)
 = \left(\begin{array}{c} \iup(X'')\\ \iup(Y'')\\\iup(Z'')\end{array}\right) + \left(\begin{array}{c} 0\\ -P''\cdot V\\-P''\cdot (P'\times V) \end{array}\right).\]
 Therefore
we can write down $\frac{{\rm d}}{{\rm d} s}({\sigma_{RM}})\sigma_{RM}^{-1}$ in terms of the normalized invariants, specifically,
\begin{equation}\label{QsigRMNormInvP}
\frac{{\rm d}}{{\rm d} s}({\sigma_{RM}})\sigma_{RM}^{-1}=\begin{pmatrix} 0& \iup(Y'') & \iup(Z'')\\ -\iup(Y'')  &0&0\\-\iup(Z'')  &0&0\end{pmatrix}.
\end{equation}
Inserting this into Equation \eqref{RecDiff} yields the all important recurrence formulae for the symbolic invariant differentiation of the normalized
invariants of the $P^{(n)}$.

We next consider the normalized invariants of the $V^{(n)}$, which are
 \begin{equation}\label{SSSDI}
 \sigma_{RM} V^{(n)}=(\iup(V_1^{(n)}), \iup(V_2^{(n)}), \iup(V_3^{(n)}))^T, \quad n\ge 0.
 \end{equation}
Differentiating both sides of \eqref{SSSDI} with respect to $s$ yields the recurrence formula for the invariant differentiation
of the symbolic normalized invariants of the components of $V^{(n)}$, 
 \begin{equation}\label{QsigRMNormInvV}
 \frac{{\rm d}}{{\rm d} s}  \begin{pmatrix}\iup(V_1^{(n)})\\ \iup(V_2^{(n)})\\ \iup(V_3^{(n)})\end{pmatrix} =
 \frac{{\rm d}}{{\rm d} s}({\sigma_{RM}})\sigma_{RM}^{-1}  \begin{pmatrix}\iup(V_1^{(n)})\\ \iup(V_2^{(n)})\\ \iup(V_3^{(n)})\end{pmatrix} 
 + \begin{pmatrix}\iup(V_1^{(n+1)})\\ \iup(V_2^{(n+1)})\\ \iup(V_3^{(n+1)})\end{pmatrix}. \end{equation}
 Setting $n=0$ into this, and since $\sigma_{RM} V = (0, 1, 0)^T$ we have that
 \begin{equation}\label{VSymbInvs}
 \left(\begin{array}{c} 0\\ 0\\0\end{array}\right)
 = \left(\begin{array}{c} \iup(V_1')\\  \iup(V_2')\\ \iup(V_3')\end{array}\right) + \left(\begin{array}{c} \iup(Y'')\\ 0\\0 \end{array}\right).\end{equation}
 
 Finally, we note that if we take a right orthonormal frame $\sigma_{RM}=\begin{pmatrix} P' & V & P'\times V\end{pmatrix}^T$, 
 where we have momentarily relaxed the differential equation condition on $V$, calculate $\sigma_{RM}'\sigma_{RM}^{-1}$ and write the components in terms
 of the normalized invariants using the Replacement Rule, Theorem \ref{reprule}, we obtain
 \begin{equation}\label{ResQsV3NotZero} \sigma_{RM}'\sigma_{RM}^{-1} = \begin{pmatrix} 0& \iup(Y'') & \iup(Z'')\\ -\iup(Y'')&0&\iup(V_3')\\ -\iup(Z'') & -\iup(V_3') & 0\end{pmatrix}.\end{equation}
 We thus see that $(2,3)$-component of $\sigma_{RM}'\sigma_{RM}^{-1}$ being zero, 
 which is what makes $\sigma_{RM}$ a \RM frame, yields a  constraint on the symbolic invariant $\iup(V_3') $.  
 The invariantization of the differential equation for $V$ yields
 \[\begin{pmatrix}\iup(V_1') \\ \iup(V_2')\\ \iup(V_3')\end{pmatrix} = -\iup(Y'')\begin{pmatrix}1 \\ 0\\ 0\end{pmatrix}.\]
 Using calculations similar to those above, it can be seen that the first two components of this equation relate to the orthonormality of $V$
 with respect to $P'$. We thus make the following definition.
 \begin{defn}[\RM  frame constraint]\label{NFproxy}
 The equation $\iup(V_3')=0$ is denoted as the \emph{\RM frame constraint}.
 \end{defn}
 
 When deriving the differential syzygy needed in the sequel, we will write the (reduced) curvature matrix with respect to 
 $s$ for the \RM frame as
 \begin{equation}\label{QsNnoCon}
\frac{{\rm d}}{{\rm d} s}({\sigma_{RM}})\sigma_{RM}^{-1}  =\left(\begin{array}{ccc} 0 & \iup(Y'') & \iup(Z'')\\ -\iup(Y'') & 0 &\iup(V_3')\\ -\iup(Z'') & -\iup(V_3')&0\end{array}\right),\qquad \iup(V_3')=0.
 \end{equation}
This is because we need to calculate the evolution of  $\iup(V_3')$  with respect to time, for our application.

\subsection{The time evolution of the frame}

We now suppose that our curve $s\mapsto P(s)$ evolves in time. The time derivatives of our variables are denoted as
\[ \frac{{\rm d}}{{\rm d}t} P^{(n)} = P_t^{(n)}, \qquad \frac{{\rm d}}{{\rm d}t} V^{(n)} = V_t^{(n)}\]
 and the action is, for 
$g=(R,\mathbf{a})\in SO(3)\ltimes \mathbb{R}$, and all $n\ge 0$,
\[ \begin{array}{rcl} P_t^{(n)}\mapsto g\cdot  P_t^{(n)} &=&R P_t^{(n)} \\
V_t^{(n)}\mapsto g\cdot  V_t^{(n)}&=&R V_t^{(n)}.
\end{array}\]

The normalized differential invariants are the components of
\[ \iup(P^{(n)}_t) = \sigma_{RM} P^{(n)}_t,\qquad \iup(V^{(n)}_t) = \sigma_{RM} V^{(n)}_t,\quad n=0,1,2,\dots \]
The curvature matrix for the extended \RM frame, with respect to time, is
\begin{equation}\label{ExCurvMxTime}
\begin{array}{rcl}
\frac{{\rm d}}{{\rm d} t}\rho_{RM} \rho_{RM}^{-1}&=&\begin{pmatrix} \frac{{\rm d}}{{\rm d} t}\sigma_{RM}\, \sigma_{RM}^{-1} & -\sigma_{RM} P_t\\
0&0\end{pmatrix}\\[15pt] &=&
\begin{pmatrix}\frac{{\rm d}}{{\rm d} t}\sigma_{RM}\, \sigma_{RM}^{-1} &
\begin{matrix} -\iup(X_t)\\ -\iup(Y_t)\\-\iup(Z_t)\end{matrix}\\0&0\end{pmatrix}.
\end{array}\end{equation}

Calculating the invariant matrix $\frac{{\rm d}}{{\rm d} t}({\sigma_{RM}})\sigma_{RM}^{-1} \in \mathfrak{so}(3)$ yields
\[\begin{array}{rcl} 
\frac{{\rm d}}{{\rm d} t}({\sigma_{RM}})\sigma_{RM}^{-1} &=&\left( \begin{array}{ccc} 0& P'_t\cdot V & P'_t\cdot (P'\times V)\\
-P'_t V & 0 & V_t\cdot (P'\times V)\\
-P'_t\cdot (P'\times V) & -V_t\cdot (P'\times V)&0\end{array}\right)\\[25pt]
&=&\left( \begin{array}{ccc} 0& \iup(Y'_t) & \iup(Z'_t)\\
-\iup(Y'_t) & 0 & \iup(V'_{3,t})\\
-\iup(Z'_t)& - \iup(V'_{3,t})&0\end{array}\right)\end{array}
\]
where we have used the Replacement Rule, Theorem \ref{reprule}, recalling $\sigma_{RM}P'=\begin{pmatrix} 1&0&0\end{pmatrix}^T$ and
$\sigma_{RM}V=\begin{pmatrix} 0&1&0\end{pmatrix}^T$.

Differentiating both sides of $\sigma_{RM} P' = (1,0,0)^T$ with respect to $t$ yields
\[\sigma_{RM} P'_t + \left(\frac{{\rm d}}{{\rm d} t}({\sigma_{RM}})\sigma_{RM}^{-1} \right) \left(\sigma_{RM} P'\right) =\begin{pmatrix}0\\0\\0\end{pmatrix}\]
so that indeed,
\[ \begin{pmatrix}\iup(X'_t)\\ \iup(Y'_t)\\ \iup(Z'_t)\end{pmatrix}= \begin{pmatrix}0\\ P'_t\cdot V\\ P'_t\cdot (P'\times V)\end{pmatrix}.\]

Further, differentiating both sides of $\sigma_{RM} V = (0,1,0)^T$ with respect to $t$ yields
\[ \sigma_{RM} V_t +\left(\frac{{\rm d}}{{\rm d} t}({\sigma_{RM}})\sigma_{RM}^{-1} \right) \left(\sigma_{RM} V\right) =\begin{pmatrix}0\\0\\0\end{pmatrix}\] so that
\[ \begin{pmatrix}\iup(V_{1,t})\\ \iup(V_{2,t})\\ \iup(V_{3,t})\end{pmatrix}= \begin{pmatrix}-P'_t\cdot V \\0\\ V_t\cdot (P'\times V)\end{pmatrix}.\]

\subsection{The syzygy operator $\mathcal{H}$}

Recall the extended \RM frame, $\rho_{RM}$, and the curvature matrices, $Q^s=\rho_{RM}'\rho_{RM}^{-1}$,
$Q^t=\frac{{\rm d}}{{\rm d} t}\rho_{RM}\rho_{RM}^{-1}$, Equations (\ref{rhoRMdef}), (\ref{ExCurvMx}), (\ref{ExCurvMxTime}) and 
repeated here for convenience,
\begin{equation}\label{rhoRMdefN} \rho_{RM}=\begin{pmatrix} \sigma_{RM} & -\sigma_{RM} P \\ 0 & 1\end{pmatrix},\end{equation}
\begin{equation}\label{ExCurvMxN} Q^s=\rho_{RM}'\rho_{RM}^{-1} = \begin{pmatrix} \sigma_{RM}'\sigma_{RM}^{-1} & \begin{matrix} -\iup(X')\\ \phantom{-}0\\ \phantom{-} 0\end{matrix}\\ 0 & \phantom{-}0\end{pmatrix}\end{equation}
where we have not yet imposed the arc length constraint $\iup(X')=1$ since we need its time evolution, and
\begin{equation}\label{ExCurvMxTimeN}
Q^t=\frac{{\rm d}}{{\rm d} t}\rho_{RM} \rho_{RM}^{-1}=
\begin{pmatrix}\frac{{\rm d}}{{\rm d} t}\sigma_{RM}\, \sigma_{RM}^{-1} &
\begin{matrix} -\iup(X_t)\\ -\iup(Y_t)\\-\iup(Z_t)\end{matrix}\\0&0\end{pmatrix}.
\end{equation}

The non-constant components of $Q^s$ are the generating invariants of the algebra of invariants of the form
$F=F(P,P',P'',\dots, V,V',V'',\dots)$; every invariant of this form can be written as a function of $\iup(Y'')$, $\iup(Z'')$ and their
derivatives with respect to $s$.  

The syzygy operator $\mathcal{H}$ that we need for our calculations in the Calculus of Variations, 
relates the time derivatives
of these generating invariants to the $s$ derivatives of the components of $\iup(P_t)$ and $\iup(V_t)$, occurring in $Q^t$.
In our case here, the syzygy operator $\mathcal{H}$ can be calculated by examining the components of the compatibility condition
 of the curvature matrices $ Q^s$ and $Q^t$, 
\begin{equation}\label{CCrhoRM} \frac{{\rm d}}{{\rm d} t} Q^s - \frac{{\rm d}}{{\rm d} s}Q^t = \left[ Q^t,  Q^s\right] \end{equation}
which follows from the fact the derivatives with respect to $t$ and $s$ commute (see \cite{Mansfield:2010aa}, \S 5.2).
We use $\sigma_{RM}'\sigma_{RM}^{-1}  $ in the form of Equation (\ref{ResQsV3NotZero}), that is, with the \RM constraint
not yet imposed, as we will need its variation with respect to time in the sequel.

Calculating the components of  Equation (\ref{CCrhoRM}) yields,
\begin{equation}\label{DiffSyz}
\begin{array}{rcl}
\frac{{\rm d}}{{\rm d} t} \iup(X') &=&\frac{{\rm d}}{{\rm d} s} \iup(X_t)-\iup(Y'')\iup(Y_t)-\iup(Z'')\iup(Z_t),\\
\frac{{\rm d}}{{\rm d} t} \iup(Y'') &=& \frac{{\rm d}^2}{{\rm d} s^2} \iup(Y_t)+\frac{{\rm d}}{{\rm d} s}\left(\iup(Y'')\iup(X_t)\right)+\iup(V_{3,t})\iup(Z''),\\
\frac{{\rm d}}{{\rm d} t} \iup(Z'') &=& \frac{{\rm d}^2}{{\rm d} s^2} \iup(Z_t)+\frac{{\rm d}}{{\rm d} s}\left(\iup(Z'')\iup(X_t)\right)-\iup(V_{3,t})\iup(Y''),\\
\frac{{\rm d}}{{\rm d} t} \iup(V_3')&=&\frac{{\rm d}}{{\rm d} s}\iup(V_{3,t}) +\iup(Y'')\frac{{\rm d}}{{\rm d} s}\iup(Z_t)-\iup(Z'')\frac{{\rm d}}{{\rm d} s}\iup(Y_t)
\end{array}\end{equation}
or in the form we require,
\[ \frac{{\rm d}}{{\rm d} t}\left(\begin{array}{c} \iup(X') \\ \iup(Y'') \\  \iup(Z'') \\
  \iup(V_3') \end{array}\right) = \mathcal{H}
\left(\begin{array}{c}\iup(X_t) \\ \iup(Y_t) \\\iup(Z_t)\\ \iup(V_{3,t})\end{array}\right)\]
where
\begin{equation}\label{Hdef}
\mathcal{H}=\left(\begin{array}{cccc}
\frac{{\rm d}}{{\rm d} s} & -\iup(Y'') & \iup(Z'') & 0 \\
 \iup(Y'')\frac{{\rm d}}{{\rm d} s}+\frac{{\rm d}}{{\rm d} s}\iup(Y'') & \frac{{\rm d}^2}{{\rm d} s^2} & 0  & \iup(Z'')\\
 \iup(Z'')\frac{{\rm d}}{{\rm d} s}+\frac{{\rm d}}{{\rm d} s}\iup(Z'') & 0 & \frac{{\rm d}^2}{{\rm d} s^2} & -\iup(Y'')\\
0 & -\iup(Z'')\frac{{\rm d}}{{\rm d} s} & \iup(Y'')\frac{{\rm d}}{{\rm d} s} &\frac{{\rm d}}{{\rm d} s}
\end{array}\right).
\end{equation}
We note that  $\mathcal{H}$ is an invariant, linear, matrix differential operator.

\section{Invariant calculus of variations}\label{ELsection}

We consider an $SE(3)$ invariant Lagrangian of the form
\[ \mathcal{L}[X',Y',Z',X'',Y'',Z'',...] = \int L(\kappa_1, \kappa_2,\kappa_{1,s},\kappa_{2,s},...) + \mu \zeta + \lambda(\eta -1)\, {\rm d}s\]
where we have set $\zeta=\iup(V_3')$, $\eta = \iup(X')$, $\kappa_1 = \iup(Y'')$ and $\kappa_2=\iup(Z'')$, and  
where $\mu$ and $\lambda$ are Lagrange multipliers for the \RM  frame constraint (Definition \ref{NFproxy}) and the arc-length constraint (Definition \ref{arclengCon}) respectively.  

Recall the Euler operator with respect to a dependent variable $u$ is defined by
\[ \mathrm{E}^u(L) = \sum_n (-1)^n \frac{{\rm d}^n}{{\rm d}s^n} \frac{\partial L}{\partial u^{(n)}} \]
where $u^{(n)} = \frac{{\rm d}^n}{{\rm d}s^n} u$.
We apply the invariantized version of the calculation of the Euler--Lagrange equations  (see  \cite{Mansfield:2010aa}, \S 7.3, also \cite{GonMan2}), 
to obtain
\[ 0=\begin{pmatrix} \mathrm{E}^X \\ \mathrm{E}^Y \\ \mathrm{E}^Z \\ \mathrm{E}^{V_3} \end{pmatrix}=\mathcal{H}^* \begin{pmatrix} \mathrm{E}^{\eta}\\ \mathrm{E}^{\kappa_1} \\ \mathrm{E}^{\kappa_2} \\ \mathrm{E}^{\zeta}\end{pmatrix}\]
that is, 
\begin{align}\label{eulerlagrangeequations}
0=\mathrm{E}^X\,&=-\kappa_1\frac{{\rm d}}{{\rm d} s} \eukao -\kappa_2\frac{{\rm d}}{{\rm d} s} \eukat - \lambda_s, \\
0=\mathrm{E}^{Y}\,&=\frac{{\rm d}^2}{{\rm d} s^2} \eukao +\frac{{\rm d}}{{\rm d} s}\left( \kappa_2 \mu \right)- \kappa_1 \lambda ,\\
0=\mathrm{E}^{Z}\,&=\frac{{\rm d}^2}{{\rm d} s^2} \eukat -\frac{{\rm d}}{{\rm d} s} \left( \kappa_1 \mu \right) - \kappa_2 \lambda ,\\
0=\mathrm{E}^{V_3}&=\eukao \kappa_2- \eukat \kappa_1 - \mu_s.
\end{align}

\begin{rem}
Note that
\[
-\kappa_1\frac{{\rm d}}{{\rm d} s}\eukao -\kappa_2\frac{{\rm d}}{{\rm d} s}\eukat = -\frac{{\rm d}}{{\rm d} s}\left( \kappa_1 \eukao + \kappa_2 \eukat\right) +\kappa_{1,s} \eukao+\kappa_{2,s} \eukat.
\]
Also, by arguments similar to that of equation $(7.17)$ in \cite{Mansfield:2010aa} we have
\begin{align*}
& \kappa_{1,s} \eukao +\kappa_{2,s} \eukao \\
&\qquad= \frac{{\rm d}}{{\rm d} s} \left(L - \sum_{m=1} \sum^{m-1}_{k=0} {\left(-1\right)}^k \left(\left(\frac{{\rm d}^k}{{\rm d} s^k} \frac{\partial L}{\partial \kappa_{1,m}}\right)\kappa_{1,m-k}
+ \left(\frac{{\rm d}^k}{{\rm d} s^k} \frac{\partial L}{\partial \kappa_{2,m}}\right)\kappa_{2,m-k}\right)\right).
\end{align*}

Therefore, $\lambda_s$ is a total derivative and we obtain
\begin{equation}\label{lambda}
\lambda =  - \kappa_1 \eukao - \kappa_2 \eukat +  L - \sum_{m=1} \sum^{m-1}_{k=0} {\left(-1\right)}^k \left(\left(\frac{{\rm d}^k}{{\rm d} s^k} \frac{\partial L}{\partial \kappa_{1,m}}\right)\kappa_{1,m-k}
- \left(\frac{{\rm d}^k}{{\rm d} s^k} \frac{\partial L}{\partial \kappa_{2,m}}\right)\kappa_{2,m-k}\right)
\end{equation}
where the constant of integration has been absorbed into $\lambda$ by Remark $7.1.9$ of \cite{Mansfield:2010aa}. This result for $\lambda$ relates
to the invariance of the Lagrangian under translation in $s$, that is, we have invariance under $s\mapsto s+\epsilon$ and hence a 
corresponding Noether law.
\end{rem}

To apply the results of \cite{GonMan2} to obtain the Noether conservation laws, we need to calculate the infinitesimals of our group action,
its associated \textit{matrix of infinitesimals}, and the right Adjoint action of the Lie group $SE(3)$ on the infinitesimal vector fields. Here we give the basic definitions, for completeness. For the Lie group $SE(3)$ and the left linear action, the precise calculations appear in \cite{GonMan2} with the  end results needed for our case here recorded in the proof of the following Theorem.
Elements in the Lie group $SE(3)$ are, in a neighbourhood of the identity element, described by six parameters,  three translation parameters,
$a$, $b$ and $c$, and
 three rotation parameters, $\theta_{xy}$, $\theta_{yz}$ and $\theta_{xz}$ where $\theta_{xy}$ is the (anticlockwise) rotation in the $(x,y)$-plane,
 and similarly for $\theta_{yz}$ and $\theta_{xz}$. For a point with coordinates $(X,Y,Z, X',Y', Z',V_1,V_2,V_3,\dots)=(z_1, z_2, \dots)$, we define the infinitesimal vector field with respect to the group parameter $a_i$ to be
 \[ \mathbf{v}_{a_i} = \sum_j \frac{\partial g\cdot z_j}{\partial a_i}\Big\vert_{g=e}\, \partial_{z_j}\]
 where $e$ is the identity element of the group. The \textit{matrix of infinitesimals} is then the matrix
 \[ \Phi=(\phi_{ij}),\qquad  \phi_{ij} = \frac{\partial g\cdot z_j}{\partial a_i}\Big\vert_{g=e}\]
 and the invariantized matrix of infinitesimals is
 \[ \Phi(I)=(\iup(\phi_{ij})).\] 
 Finally, given an infinitesimal vector field $\mathbf{v}=\sum_j \xi^j(z)\partial_{z_j}$, the right Adjoint action of $G$ on $\mathbf{v}$ is given by 
 \[ \sum_j \xi^j(z)\partial_{z_j}\mapsto \sum_j \xi^j(g\cdot z)\partial_{g\cdot z_j}.\]
 This determines a linear map, $\mathbf{v}_{a_i}\mapsto \sum_{k}(\mathcal{A}d(g) )_{ik}\mathbf{v}_{a_k}$ called the right Adjoint action of 
 $G$ on its Lie algebra of vector fields. (See \cite{Mansfield:2010aa} for further details.)
 
Continuing to apply the results of  \cite{GonMan2}, we obtain that Noether's laws are as given in the following theorem,
\begin{theorem} The conservation laws are of the form
\begin{equation}\label{NCL}
\left(\begin{array}{cccccc} \sigma_{RM}^T & 0 \\ D {\bf X} \sigma_{RM}^T & D \sigma_{RM}^T D \end{array}\right)\left( \begin{array}{c} \lambda \\ -\frac{{\rm d}}{{\rm d} s} \eukao - \mu \kappa_2 \\ -\frac{{\rm d}}{{\rm d} s} \eukat + \mu \kappa_1 \\ \mu \\ \eukat \\ \eukao \end{array}\right) = \left( \begin{array}{c} c_1 \\ c_2 \\ c_3 \\ c_4 \\ c_5 \\ c_6 \end{array}\right)
\end{equation}
where
\[
{\bf X}=\left(\begin{array}{ccc} 0 & -Z & Y\\Z&0&-X\\ -Y& X &0\end{array}\right)
 \]
 $D=\textrm{diag}(1,-1,1)$, and the $c_i$ are constants.
\end{theorem}
\begin{proof} 
In order to compute the conservation laws, we need the boundary terms $\mathcal{A}_{\mathcal{H}}$, the (right) Adjoint representation
of the frame $\rho_{RM}$ and the invariantized matrix of infinitesimals, which we defined above. We now consider these in turn.

Let $\mathrm{E}(L)=\begin{pmatrix} \mathrm{E}^{\eta} & \mathrm{E}^{\kappa_1} & \mathrm{E}^{\kappa_2} & \mathrm{E}^{\zeta} \end{pmatrix}$ and let
$\phi^t=\begin{pmatrix} \iup(X_t) & \iup(Y_t) & \iup(Z_t) & \iup(V_{3,t})\end{pmatrix}^T$. Then
the  boundary terms $\mathcal{A}_{\mathcal{H}} $ are defined by 
\[ \frac{{\rm d}}{{\rm d} s} \mathcal{A}_{\mathcal{H}} = \mathrm{E}(L) \mathcal{H}\phi^t - \mathcal{H}^* \mathrm{E}(L) \phi^t.\]
By direct calculation, we obtain
\begin{align*}
\mathcal{A}_{\mathcal{H}}&= \lambda \iup(X_t) +  \left( -\frac{{\rm d}}{{\rm d} s} \eukao - \mu \kappa_2\right) \iup(Y_t) + \left( -\frac{{\rm d}}{{\rm d} s} \eukat + \mu \kappa_1\right) \iup(Z_t) \\ & \qquad  + \eukao \iup(Y'_t) + \eukat \iup(Z'_t) + \mu \iup(V_{3,t})\\
&= \mathcal{C}^X \iup(X_t) + \mathcal{C}^Y \iup(Y_t) + \mathcal{C}^Z  \iup(Z_t) + \mathcal{C}^{Y'}\iup(Y'_t)+ \mathcal{C}^{Z'}\iup(Z'_t)
+\mathcal{C}^{V_{3,t}} \iup(V_{3,t})
\end{align*}
where this defines the coefficients $\mathcal{C}$ and where we have used the syzygies
\begin{align*}
\iup(Y'_t) & = \frac{{\rm d}}{{\rm d} s} \iup(Y_t) + \kappa_1 \iup(X_t),\\
\iup(Z'_t) & = \frac{{\rm d}}{{\rm d} s} \iup(Z_t) + \kappa_2 \iup(X_t)
\end{align*} 
 to eliminate derivatives of $\iup(Y_t)$ and $\iup(Z_t)$ in the boundary terms.

In \cite{GonMan2}, the authors show the (right) Adjoint representation of ${\rm S}E(3)$  with respect to the generating infinitesimal vector fields of the action,
\begin{equation}\label{VOI}
\begin{array}{lll}
{\bf v}_{a} = \partial_X, & {\bf v}_{b} = \partial_Y, & {\bf v}_{c}  =  \partial_Z, \\ {\bf v}_{YZ} = Y \partial_Z -  Z \partial_Y, & {\bf v}_{XZ}  =  X \partial_Z - Z \partial_X, & {\bf v}_{XY} = X \partial_Y -Y \partial_X\end{array}
\end{equation}
is of the form, for $g=(R,\mathbf{a})$,
\[
\mathcal{A}d(g)=\left(\begin{array}{cc} R & 0 \\ DAR & DRD \end{array}\right)
\]
where $R \in {\rm S}O(3)$,  $D$ is the diagonal matrix $D=\mbox{diag}(1,-1,1)$ and $A$ is the matrix
\[
 A=\left(\begin{array}{ccc} 0 & -c & b\\c&0&-a\\ -b& a &0\end{array}\right)
\]
where $\mathbf{a}={(a,b,c)}^{T}$ is the translation vector component of $g$.

Hence
\[
{\mathcal{A}d(\rho_{RM})}^{-1}=\left(\begin{array}{cccccc} \sigma_{RM}^{T} & 0 \\ D {\bf X} \sigma_{RM}^{T} & D \sigma_{RM}^{T} D \end{array}\right)
\]
where
\[
{\bf X}=\left(\begin{array}{ccc} 0 & -Z & Y\\Z&0&-X\\ -Y& X &0\end{array}\right).
 \]
 
The invariantized matrix of infinitesimals with respect to the basis \eqref{VOI} is  
 \[ \Phi(I) =\bordermatrix{ & X & Y & Z & Y' & Z' & V_3
\cr a & 1 & 0 & 0 & 0 & 0 & 0
\cr b &0&1&0& 0 & 0 & 0
\cr c & 0&0&1& 0 & 0 & 0
\cr\theta_{yz} & 0 & 0 & 0& 0 & 0 & 1
\cr \theta_{xz} &0&0&0& 0 & 1 & 0
\cr\theta_{xy} & 0&0&0& 1 & 0 & 0
\cr}.\]

Finally, the conservation laws obtained via Noether's theorem for the unidimensional case are, see \cite{GonMan2},
\begin{equation}\label{conservatiolaws}
{\mathcal{A}d(\rho)}^{-1} v(I)={\bf c}
\end{equation}
where
\begin{equation}\label{infinitesimals}
v(I)=\sum_{\alpha} \Phi^{\alpha}(I) \mathcal{C}^{\alpha}={\left( \begin{array}{cccccc} \lambda & -\frac{{\rm d}}{{\rm d} s} \eukao - \mu \kappa_2 & -\frac{{\rm d}}{{\rm d} s} \eukat + \mu \kappa_1 & \mu & \eukat & \eukao \end{array}\right)}^T
\end{equation}
as required.
\end{proof}

\begin{rem} A quick check on this result is obtained by noting the following.
Differentiating  \eqref{NCL} with respect to $s$ and multiplying by $\mathcal{A}d(\rho_{RM})$, we get
\[
\frac{{\rm d}}{{\rm d} s}v(I)=\frac{{\rm d}}{{\rm d} s}\left(\mathcal{A}d(\rho_{RM}) \right){\mathcal{A}d(\rho)}^{-1} v(I)
\] 
i.e,
\begin{equation}\label{diffv}
\frac{{\rm d}}{{\rm d} s} v(I) = \left( \begin{array}{cccccc} 0 & \kappa_1 & \kappa_2 & 0 & 0 & 0 \\ -\kappa_1 & 0 & 0 & 0 & 0 & 0 \\ -\kappa_2 & 0 & 0 & 0 &0 & 0\\ 0 & 0 & 0 & 0 & -\kappa_1 & \kappa_2 \\ 0 & 0 & -1 & \kappa_1 & 0 & 0 \\ 0 & -1 & 0 & -\kappa_2 & 0 & 0 \end{array}\right) v(I).
\end{equation}
We observe that the first four rows are equivalent to the Euler-Lagrange equations while last two rows are identically 0, as expected.
\end{rem}


\section{Solution of the integration problem}\label{IPsection}

The conservation laws \eqref{NCL} can reduce the integration problem. We write these in the form,

\begin{equation}\label{NCLinv}
\left(\begin{array}{cccccc} \sigma^T_{RM} & 0 \\ D {\bf X} \sigma^T_{RM} & D \sigma^T_{RM} D \end{array}\right)\left(\begin{array}{c} \textbf{w}_1(I)\\ \textbf{w}_2(I)\end{array}\right)=\left(\begin{array}{c} \textbf{c}_1\\ \textbf{c}_2\end{array}\right)
\end{equation}
where $v(I)=( \textbf{w}_1(I), \textbf{w}_2(I))^T$, $\textbf{c}=(\textbf{c}_1,\textbf{c}_2)^T$ and recalling
\[
\mathbf{X}=\left(\begin{array}{ccc} 0 & -Z & Y\\Z&0&-X\\ -Y& X &0\end{array}\right).
 \]

Since $\sigma_{RM}\in {\rm S}O(3)$ we have from
\begin{equation}\label{CLeqnCompOne} \sigma_{RM} \textbf{c}_1=\textbf{w}_1(I)\end{equation} 
that
\begin{equation}\label{modcweqn}
|\textbf{c}_1|=|\textbf{w}_1(I)|.
\end{equation}
Further,  multiplying the second component of Equation (\ref{NCLinv}) on the left by $  \textbf{c}_1(I)^T D$, since $D^2=I$, we obtain
\begin{equation}\label{SecondIntELeqn}
\textbf{w}_1^TD\textbf{w}_2 = \textbf{c}_1^TD\textbf{c}_2.
\end{equation}

In order to solve Equation (\ref{CLeqnCompOne}), as far as we can, for the components of $\sigma_{RM}$ in terms of the components of
$\textbf{c}_1$ and $\textbf{w}_1(I)$, we use the Cayley representation $\Cay$ of elements of $SO(3)$. We define 
\[ \Cay(x_1,x_2,x_3,x_4)=\left(\begin{array}{ccc}
x_1^2+x_2^2-x_3^3-x_4^2 & -2(x_1x_4-x_2x_3) & 2(x_1x_3+x_2x_4)\\
2(x_1x_4+x_2x_3) & x_1^2-x_2^2+x_3^3-x_4^2 & -2(x_1x_2-x_3x_4)\\
-2(x_1x_3-x_2x_4)& 2(x_1x_2+x_3x_4) & x_1^2-x_2^2-x_3^3+x_4^2\end{array}\right).
\]
Then provided $x_1^2+x_2^2+x_3^3+x_4^2=1$, $\Cay(x_1,x_2,x_3,x_3)\in SO(3)$, has an axis of rotation $(x_2, x_3, x_4)^T$ and
the angle of rotation $\psi$ satisfies $2 x_1^2-1 = \cos\psi$. Hence we may define, for an angle $\psi$ and axis of rotation $\textbf{a}=(a_1,a_2, a_3)^T\ne 0$, 
\[ R(\psi, \textbf{a}) = \Cay\left(\cos\left(\frac{\psi}2\right), \sin\left(\frac{\psi}2\right) \frac{a_1}{|\mathbf{a}|},\sin\left(\frac{\psi}2\right) \frac{a_2}{|\mathbf{a}|},
\sin\left(\frac{\psi}2\right) \frac{a_3}{|\mathbf{a}|}\right) \in {\rm S}O(3).\]
There are two cases. 

\noindent \textbf{Case 1}.\quad  If $\mathbf{w}_1 + \mathbf{c}_1$ is bounded away from zero,
we note that $\sigma_{RM}$ may be taken to be a product of a rotation about $\mathbf{c}_1 + {(0,0,|\mathbf{c}_1|)}^{T}$ with angle $\pi$ followed by a rotation about ${(0,0,|\mathbf{c}_1|)}^{T}$ with any angle $\psi$ and a rotation about $\mathbf{w}_1 + {(0,0,|\mathbf{c}_1|)}^{T}$ with angle $\pi$, that is,
\[ \sigma_{RM}=R(\pi, \mathbf{w}_1 + {(0,0,|\mathbf{c}_1|)}^{T})R(\psi(s), {(0,0,|\mathbf{c}_1|)}^{T})R(\pi, \mathbf{c}_1 + {(0,0,|\mathbf{c}_1|)}^{T} ).\]
This solves for $\sigma_{RM}$ up to the angle $\psi$.
If we differentiate this with respect to $s$, right multiply by $\sigma_{RM}^{-1}$
\[ \sigma^{-1}_{RM}=R(\pi, \mathbf{c}_1 + {(0,0,|\mathbf{c}_1|)}^{T})R(-\psi(s), {(0,0,|\mathbf{c}_1|)}^{T})R(\pi, \mathbf{w}_1 + {(0,0,|\mathbf{c}_1|)}^{T} )\]
using \eqref{diffv} and taking into account that
\[
\frac{{\rm d}}{{\rm d} s}(\sigma_{RM})\sigma_{RM}^{-1}=\left(\begin{array}{ccc} 0 & \kappa_1 & \kappa_2 \\ -\kappa_1 & 0 & 0 \\ -\kappa_2 & 0 & 0 \end{array}\right)
\] we obtain a remarkable equation for
$\psi$, specifically,
\begin{equation}\label{PsiEqnCaseOne}
\psi_s = - \kappa_1 +
 \frac{v_2(I)}{|\mathbf{c}_1| + v_3(I)}\, \kappa_2
\end{equation}
where recall $v_2(I)$ and $v_3(I)$ are the second and third components of the vector of invariants, $v(I)$, an also, be definition,
the second and third components of $\mathbf{w}_1$.
\noindent \textbf{Case 2}.\quad  If $\mathbf{w}_1 - \mathbf{c}_1$ is bounded away from zero, we note that $\sigma_{RM}$ may be taken to be a product of a rotation about $\mathbf{c}_1 + {(0,0,-|\mathbf{c}_1|)}^{T}$ with angle $\pi$ followed by a rotation about ${(0,0,-|\mathbf{c}_1|)}^{T}$ with any angle $\psi$ and a rotation about $\mathbf{w}_1 + {(0,0,-|\mathbf{c}_1|)}^{T}$ with angle $\pi$, that is,
\[ \sigma_{RM}=R(\pi, \mathbf{w}_1 + {(0,0,-|\mathbf{c}_1|)}^{T})R(\psi(s), {(0,0,-|\mathbf{c}_1|)}^{T})R(\pi, \mathbf{c}_1 + {(0,0,-|\mathbf{c}_1|)}^{T} ).\]
Since the matrix on the right and the matrix on the left are constant, we obtain the same equation for $\psi$ as above, but with the signs of $\mathbf{c}_1$ reversed.
Hence in this case,
\begin{equation}\label{PsiEqnCaseTwo}
\psi_s = \kappa_1 +
 \frac{v_2(I)}{|\mathbf{c}_1| - v_3(I)}\, \kappa_2.
\end{equation}

In either case, we obtain $\sigma_{RM}$ up to a quadrature. There is a significant overlap in the domains of the two cases, and matching one to the other, as needed, is not a problem.

Next, we seek $P$. We note the first row of $\sigma_{RM}$ is $P'$,
and so we can always obtain $P$ by quadrature. However, we note that only one component needs to be calculated this way, as the
second component of Equation (\ref{NCLinv}) provides algebraic equations for two of the components of $P$, i.e,
\begin{align*}
X &=\frac{1}{v_3(I)} (v_4(I) + Zv_2(I) - {(\sigma D {\bf c_2})}_1), \\
Y &= \frac{1}{v_3(I)} (v_5(I) + Zv_1(I) + {(\sigma D {\bf c_2})}_2) \\
\end{align*}
where $Z$ has been solved previously by quadrature.

We conclude by noting that the conservation laws provide two first integrals of the Euler--Lagrange equations.
They may be used to solve for $P$ in terms of two quadratures, and they also solve for the normal vector $V$ in terms of
one quadrature, that of $\psi$.
Finally, we note that it is easy to obtain the \FS frame from our calculations, since it is defined in terms of $P'$ and $P''$.

\section{Examples and applications}\label{EXsection}

We  examine a Lagrangian which is not possible to study in the \FS framework. Secondly, we study functionals used to model some biological structures, invariant under ${\rm S}E(3)$ and depending on the curvature, torsion and their derivatives, but  using our results for the \RM  frame.

We first show that every Lagrangian which can be written in terms of the Euclidean curvature $\kappa$ and  torsion $\tau$ can be written
in terms of the invariants, $\kappa_1$ and $\kappa_2$.
From \eqref{k1k2} we have that
\[
\kappa_1 = \kappa\cos{\theta}, \qquad \kappa_2 = \kappa \sin{\theta}
\]
and therefore, using $\tan\theta = \kappa_2/\kappa_1$ and $\theta_s=\tau$ we have, 
\begin{equation}\label{RMtoFS}
\kappa=\sqrt{\kappa_1^2 + \kappa_2^2}, \qquad \tau = \frac{\kappa_1 \kappa_{2,s} - \kappa_{1,s}\kappa_2}{\kappa_1^2 + \kappa_2^2}.
\end{equation}

But the converse is not true. Lagrangians which depend only on $\kappa_2/\kappa_1$ cannot be written in terms of $\kappa$ and $\tau$.
Our first example is the simplest such Lagrangian, which we study simply because we can.

\subsection{Invariant Lagrangians involving only  $\kappa_2/\kappa_1$}\label{tantheta}

Let us consider the Lagrangian
\begin{align*}
\mathcal{L}[\kappa_2/\kappa_1]  
&= \int \frac{1}{2} {\left( \frac{\kappa_2}{\kappa_1}\right)}^2 + \lambda \left(\eta - 1 \right) + \mu \zeta \ {\rm d}s\\
&=\int \tan{\theta}^2 + \lambda \left(\eta - 1 \right) + \mu \zeta \ {\rm d}s
\end{align*}
where recall $\eta=1$ is the arc-length constraint and $\zeta=0$ is the \RM constraint.

Using the results of the previous section,
we obtain the Euler--Lagrange equations
\begin{equation}\label{tanEL1}
\begin{aligned}
&\left( -\frac{12 {\frac{{\rm d}}{{\rm d} s}\kappa_1}^2}{\kappa_1^5} +\frac{3 \frac{{\rm d}^2}{{\rm d} s^2}\kappa_1}{\kappa_1^4} - \frac{1}{2\kappa_1} \right) \kappa_2^2 +
 \left( \frac{12 {\frac{{\rm d}}{{\rm d} s}\kappa_1}{\frac{{\rm d}}{{\rm d} s}\kappa_2}}{\kappa_1^4} -\frac{2 \frac{{\rm d}^2}{{\rm d} s^2}\kappa_2}{\kappa_1^3} - \mu_s \right) \kappa_2 \\&
\qquad \qquad -\frac{2\frac{{\rm d}}{{\rm d} s}\kappa_2^2}{\kappa_1^3}+\mu \frac{{\rm d}}{{\rm d} s} \kappa_2 = 0,\\
 \end{aligned}
 \end{equation}
 \begin{align}
& \label{tanEL2}-\frac{\kappa_2^3}{2\kappa_1^2} + \left( \frac{6 \frac{{\rm d}}{{\rm d} s}\kappa_1^2}{\kappa_1^4} - \frac{2\frac{{\rm d}^2}{{\rm d} s^2}\kappa_1}{\kappa_1^3} \right)\kappa_2 \frac{\frac{{\rm d}^2}{{\rm d} s^2} \kappa_2}{\kappa_1^2} -\frac{4\frac{{\rm d}}{{\rm d} s}\kappa_1\frac{{\rm d}}{{\rm d} s}\kappa_2}{\kappa_1^3}-\mu_s \kappa_1 -\mu \frac{{\rm d}}{{\rm d} s}\kappa_1 = 0,\\
&\mu_s + \frac{\kappa_2^3}{\kappa_1^3} + \frac{\kappa_2}{\kappa_1} =0 
\end{align}
where $\lambda= \frac{1}{2}{\left(\frac{\kappa_2}{\kappa_1}\right)}^2$ has been solved using \eqref{lambda}.
Further, the vector of invariants $v(I)$ needed for the conservation laws is
\[
v(I)=\left( \begin{array}{c} \frac{1}{2}{\left(\frac{\kappa_2}{\kappa_1}\right)}^2 \\
-\frac{\kappa_2}{\kappa_1^4}(\kappa_1^4\mu-2\kappa_1\frac{{\rm d}}{{\rm d} s}\kappa_2 + 3\kappa_2 \frac{{\rm d}}{{\rm d} s}\kappa_1)\\
-\frac{\frac{{\rm d}}{{\rm d} s}\kappa_2}{\kappa_1^2} + \frac{2 \kappa_2\frac{{\rm d}}{{\rm d} s}\kappa_1}{\kappa_1^3} + \mu\kappa_1\\
\mu \\
\frac{\kappa_2}{\kappa_1^2}\\
-\frac{\kappa_2^2}{\kappa_1^3}
\end{array}\right).
\]

Solving \eqref{tanEL1}, \eqref{tanEL2} along with \eqref{SecondIntELeqn}, \eqref{PsiEqnCaseOne} and \eqref{PsiEqnCaseTwo} for $\kappa_1,\kappa_2,\mu$ and $\psi$ with initial conditions
\[
\begin{array}{llll}
\kappa_1(0) = 1, &\kappa_2(0) = \frac{1}{2}, &\frac{{\rm d}}{{\rm d} s}\kappa_1(0) = 1,& \frac{{\rm d}}{{\rm d} s}\kappa_2(0) = 1,\\[10pt]
\lambda(0) = 1, &\mu(0) = 1, &Z(0)=1, &\psi(0)=0
\end{array}
\]
we obtain the following solutions, see Figures \ref{App0a},   \ref{App0b},  \ref{App0c}.

\begin{figure}
\caption{Solutions for the invariants $\kappa_1$, $\kappa_2$, $\theta$ and $\kappa$\label{App0a}}
 \vspace{5mm}
\begin{tabular}{ccc}

    \includegraphics[width=0.3\textwidth]{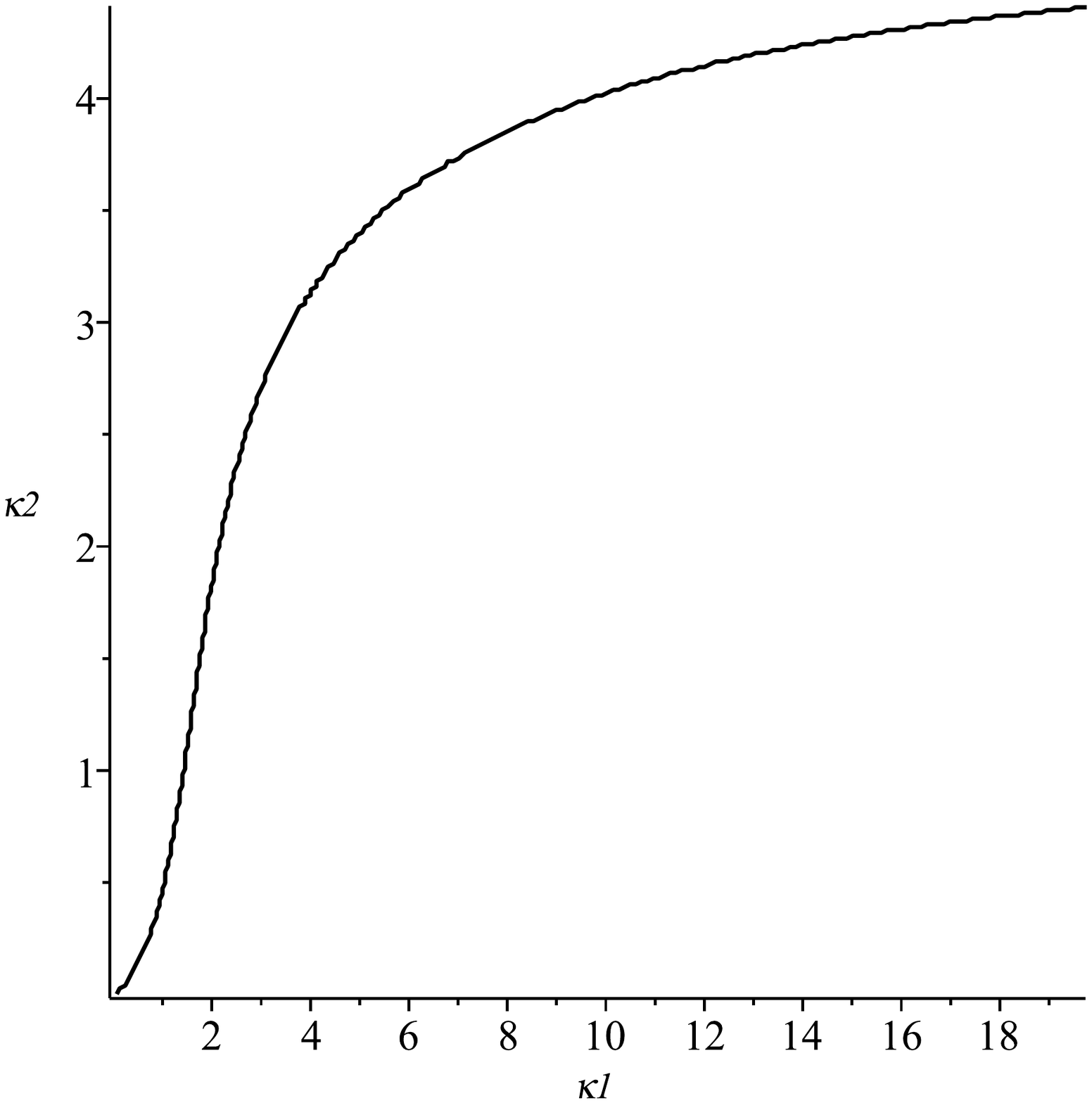} &

    \includegraphics[width=0.3\textwidth]{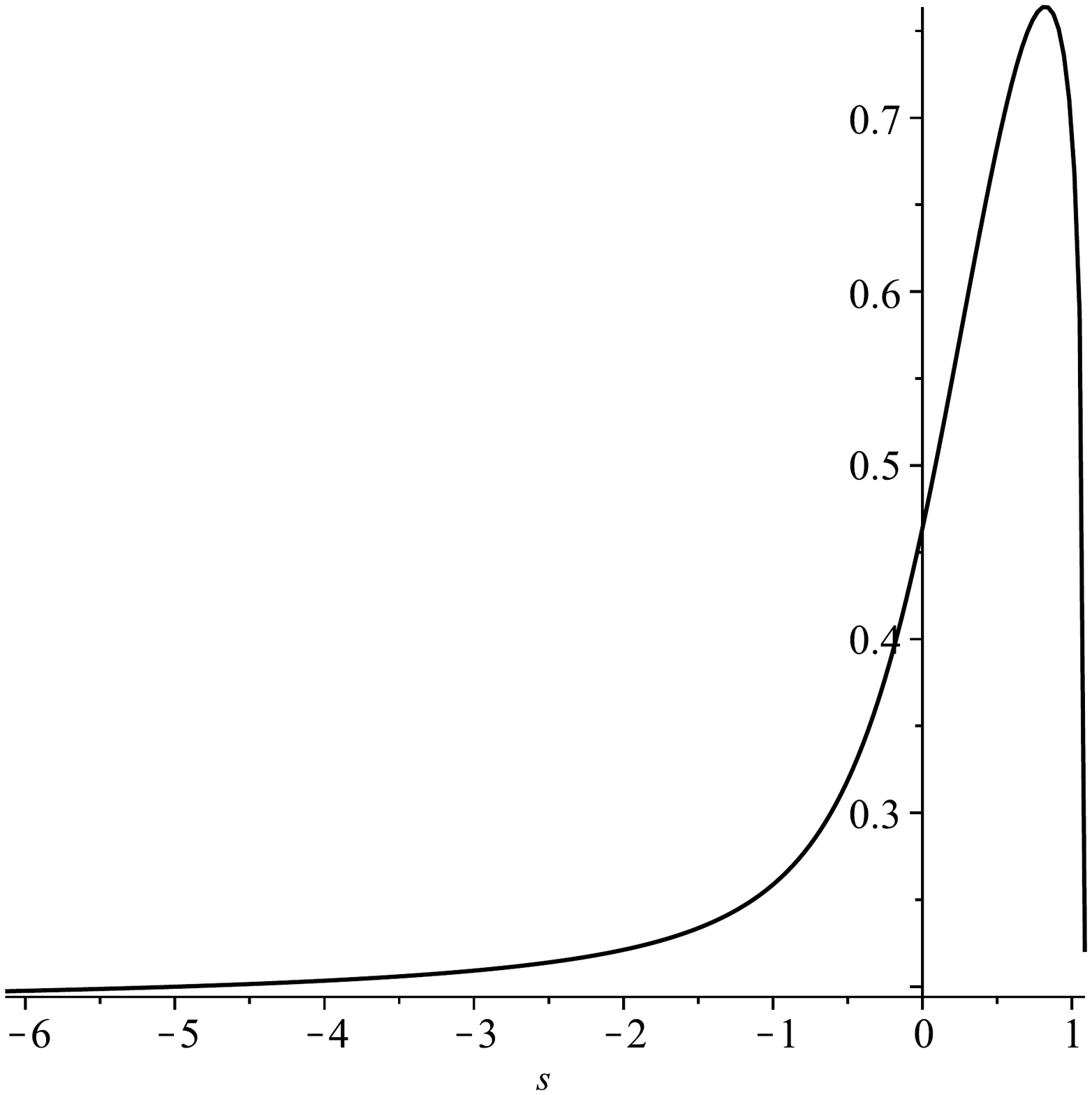} &

    \includegraphics[width=0.3\textwidth]{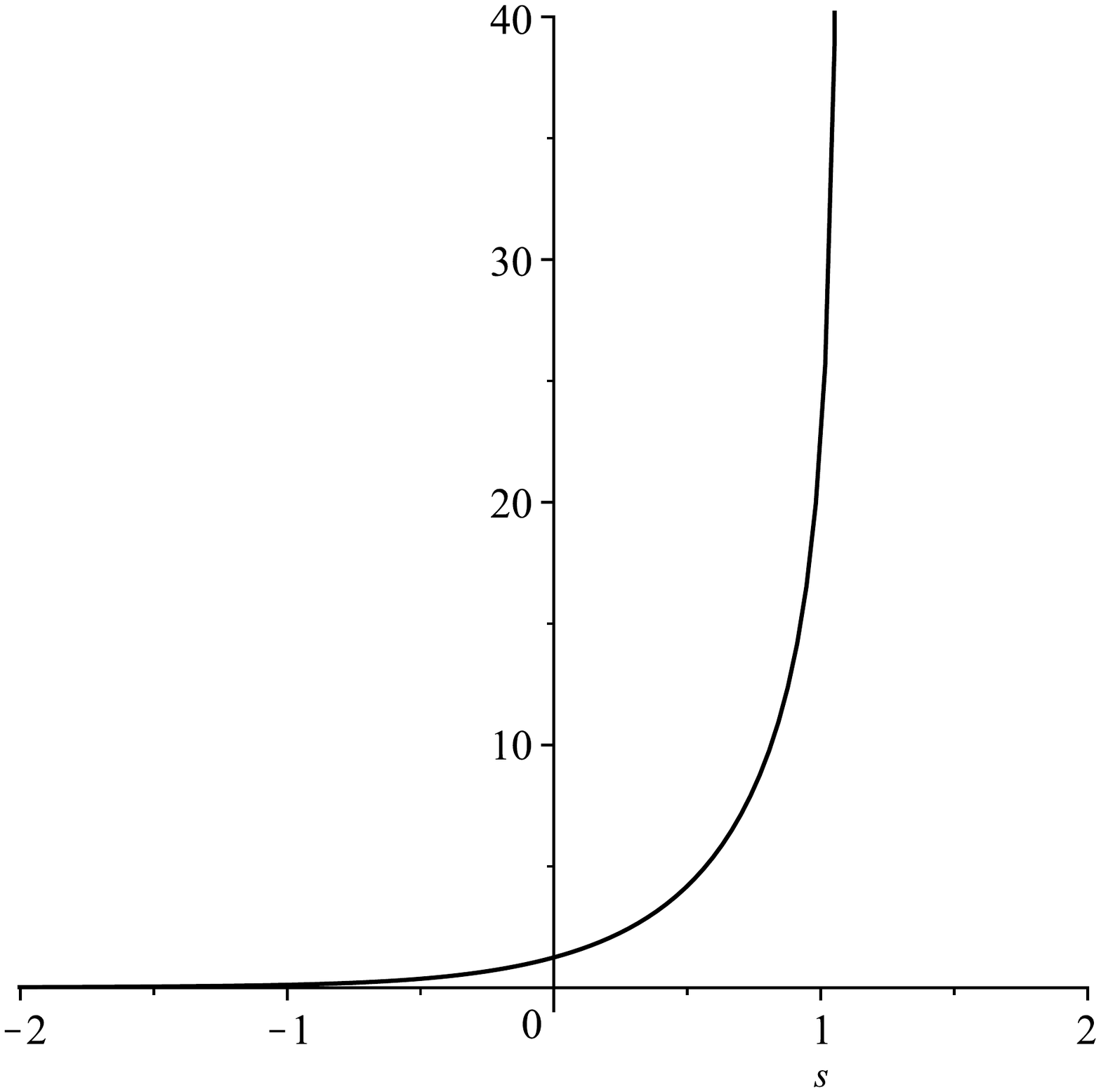}\\
     $\kappa_1$ vs $\kappa_2$ & $s$ vs $\theta$ & $s$ vs $\kappa$ \\
\end{tabular}
\end{figure}

\begin{figure}[ht!]
\begin{center}
\caption{Plots of the first integrals\label{App0b}}
 \vspace{5mm}
\begin{tabular}{cc}

    \includegraphics[width=0.3\textwidth]{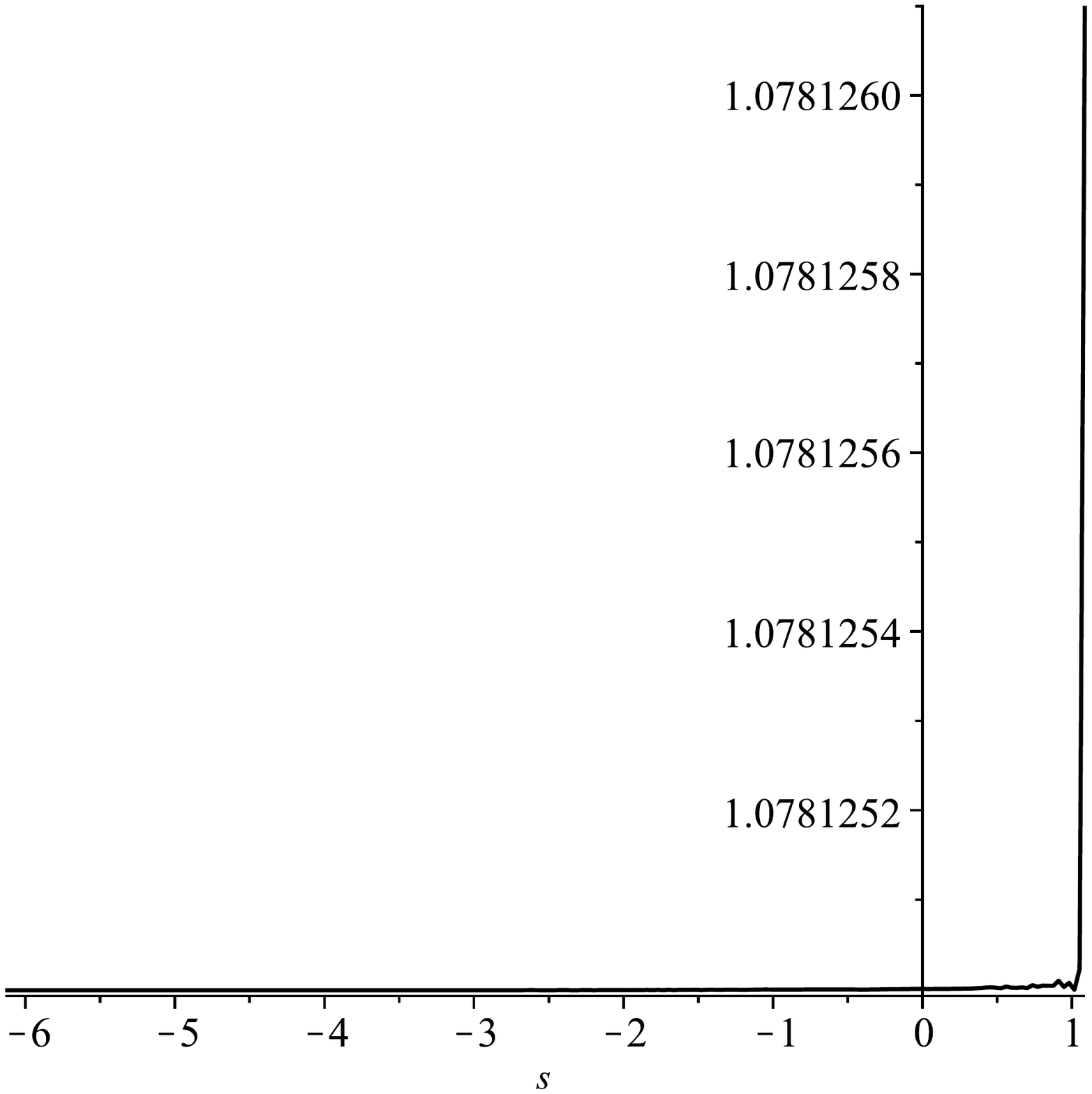} &

    \includegraphics[width=0.3\textwidth]{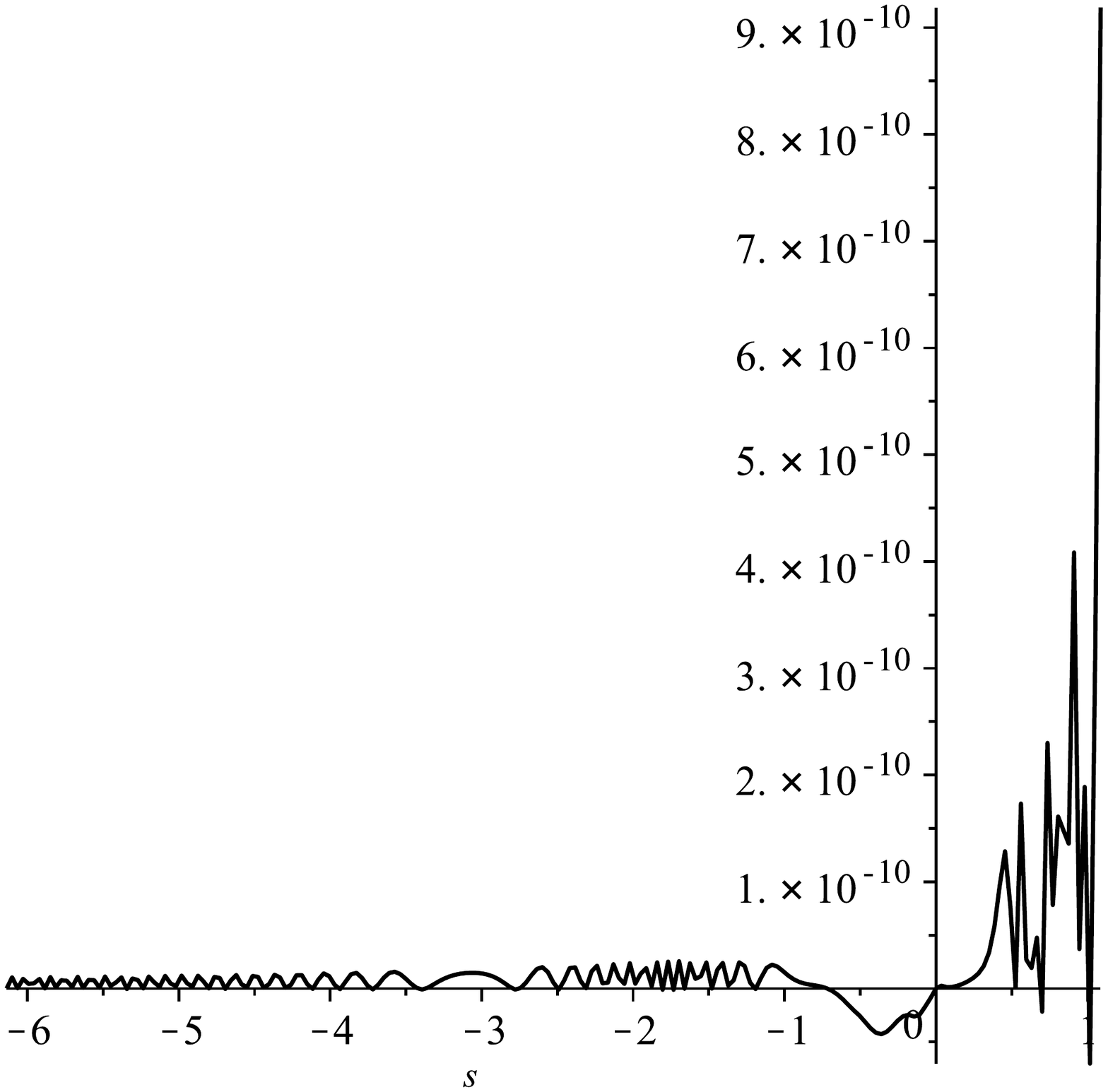} \\
     $v_1^2+v_2^2+v_3^2=c_1^2+c_2^2+c_3^3$  & $v_1v_4-v_2v_5+v_3v_6=c_1c_4-c_2c_5+c_3c_6$ \\
\end{tabular}
\end{center}
\end{figure}

\begin{figure}
    \caption{Sweep surfaces using the \RM frame and the \FS frame along the extremal curve\label{App0c}}
    \vspace{5mm}
    \centering
   \begin{tabular}{cc}

    \includegraphics[width=0.45\textwidth]{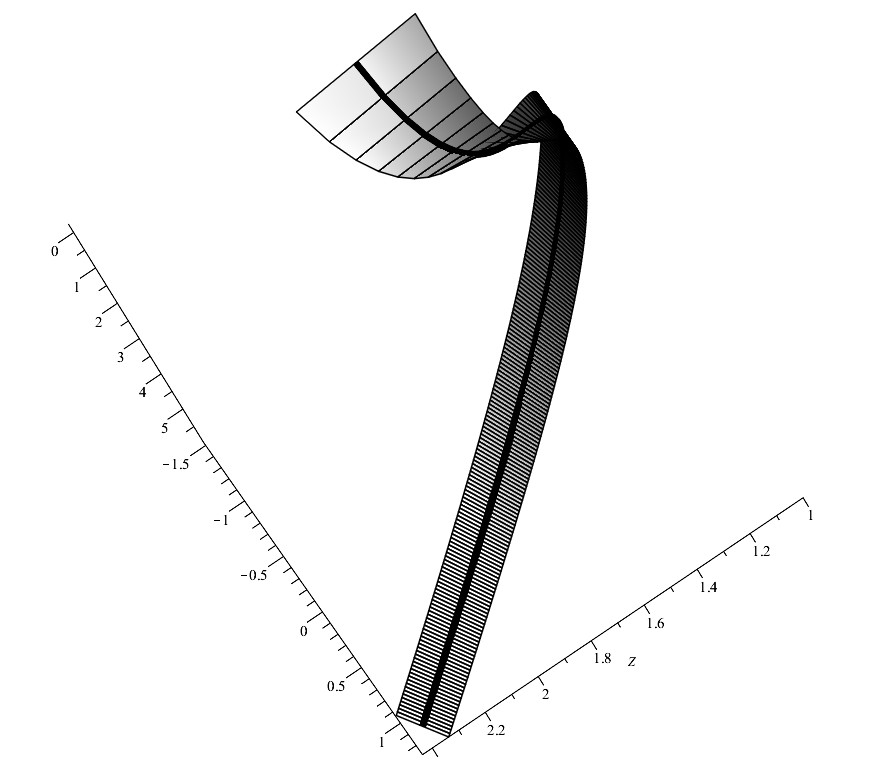} &

    \includegraphics[width=0.4\textwidth]{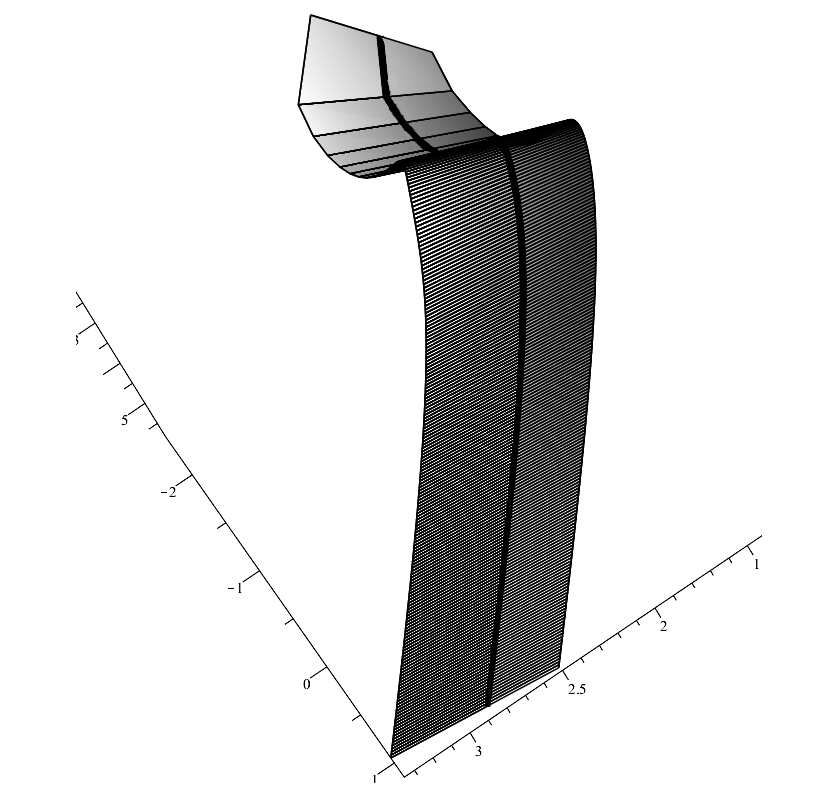} \\
   Plot of $V$ along the extremal curve  & Plot of $P''$ along the extremal curve \\
   using the \RM frame  & using the Frenet--Serret frame \\
\end{tabular}
\end{figure}

\pagebreak

\subsection{Further examples}
In order to model strands of proteins, nucleid acids and polymers,  some authors have made use of the classic calculus of variations and studied the Euler--Lagrange equations of an energy functional depending on the curvature, torsion and  their first derivatives.
In \cite{McCoy} and \cite{TMH} the authors consider protein backbones and polymers as a smooth curve in $\mathbb{R}^3$ and use the Frenet--Serret equations in order to compute a variation to the curve.  The Euler--Lagrange equations are obtained for these type of functionals. In \cite{FNS} the same method is used to obtain the Euler--Lagrange equations for functionals which are linear in the curvature.
 In this section we study two examples from the families of functionals studied, but in terms of the invariants $\kappa_1$ and $\kappa_2$. 
The conversion of a functional given in terms of Euclidean curvature and torsion to one given in terms of $\kappa_1$ and $\kappa_2$ is given in Equation \eqref{RMtoFS}.

\subsubsection{The Lagrangian $\int \kappa^2 \tau \ {\rm d}s = \int \kappa_1 \kappa_{2,s}-\kappa_{1,s}\kappa_2 \ {\rm d}s$}

For the Lagrangian 
\[
\int \kappa_1 \kappa_{2,s} - \kappa_{1,s}\kappa_2  \ {\rm d}s
\]
the Euler--Lagrange equations are
\begin{align}\label{eulerlagrangelag2a}
2 \kappa_{2,sss}+3\kappa_{2,s}\kappa^2&=0,\\\label{eulerlagrangelag2b}
-2 \kappa_{1,sss}-3\kappa_{1,s}\kappa^2&=0.
\end{align}
The conservation laws are of the form \eqref{conservatiolaws} where
\[
v(I)=(2(\kappa_{1,s}\kappa_2-\kappa_1\kappa_{2,s}) \quad  -2\kappa_{2,ss} - \kappa_2\kappa^2  \quad  -2\kappa_{1,ss} + \kappa_1\kappa^2 \quad  \kappa^2 \quad -2\kappa_{1,s}  \quad 2\kappa_{2,s} )^T.
\]

Solving \eqref{eulerlagrangelag2a}, \eqref{eulerlagrangelag2b} along with  \eqref{PsiEqnCaseOne} and \eqref{PsiEqnCaseTwo} for 
$\kappa_1,\kappa_2$ and $\psi$ with initial conditions
\[
\begin{array}{llll}
\kappa_1(0) = 1, &\kappa_2(0) = \frac{1}{2}, &\frac{{\rm d}}{{\rm d} s}\kappa_1(0) = 1,& \frac{{\rm d}}{{\rm d} s}\kappa_2(0) = 1,\\[10pt]
\frac{{\rm d}^2}{{\rm d} s^2}\kappa_1(0) = 1,& \frac{{\rm d}^2}{{\rm d} s^2}\kappa_2(0) = 1, &\psi(0)=0 &
\end{array}
\]
and integrating to obtain the extremizing curve and its \RM frame, we obtain the following solutions, see Figures \ref{App1a}, \ref{App1b}, \ref{App1c}.

\begin{figure}
\caption{Solutions for the invariants $\kappa_1$, $\kappa_2$, $\theta$ and $\kappa$\label{App1a}}
\begin{tabular}{ccc}

    \includegraphics[width=0.3\textwidth]{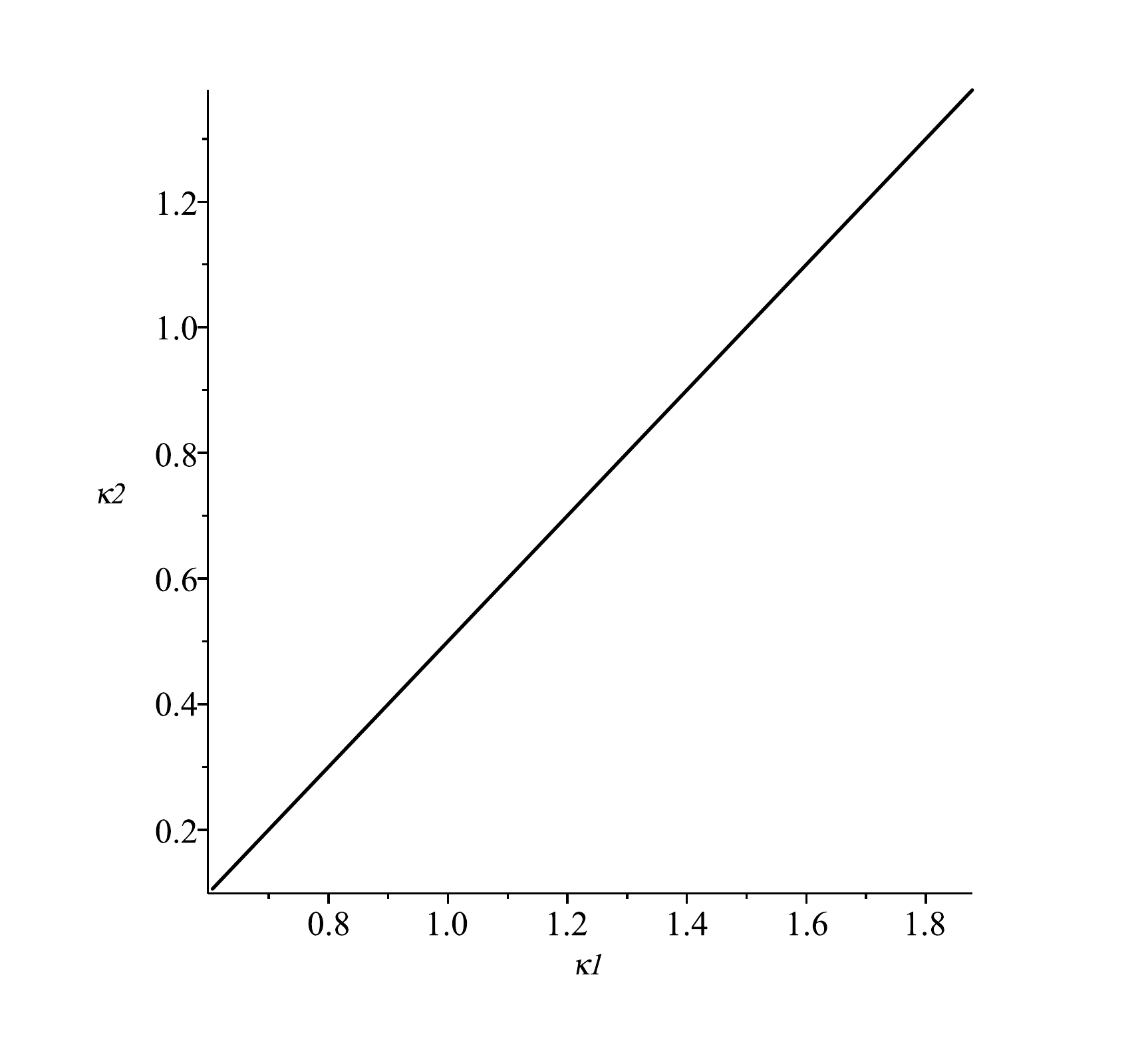} &

    \includegraphics[width=0.3\textwidth]{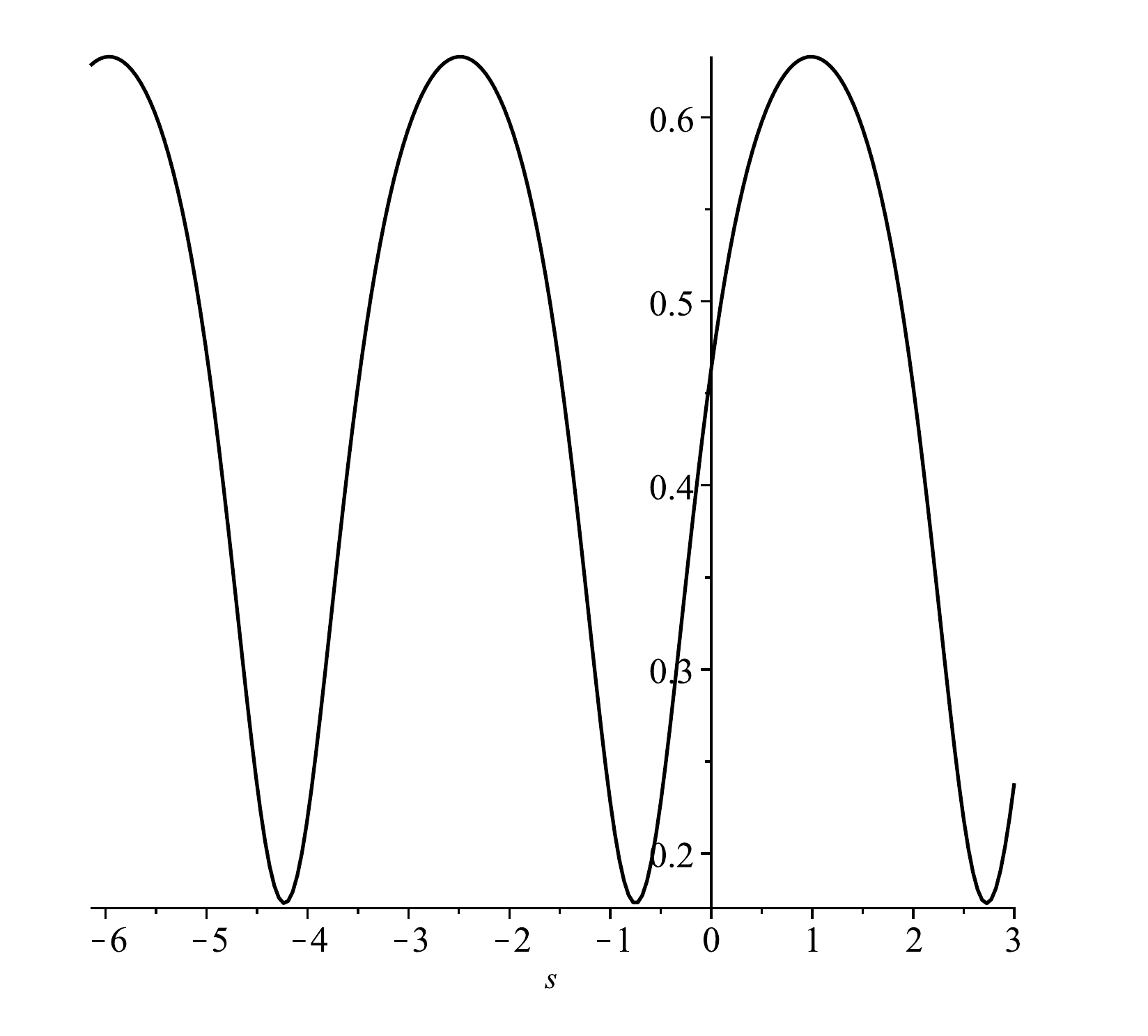} &

    \includegraphics[width=0.3\textwidth]{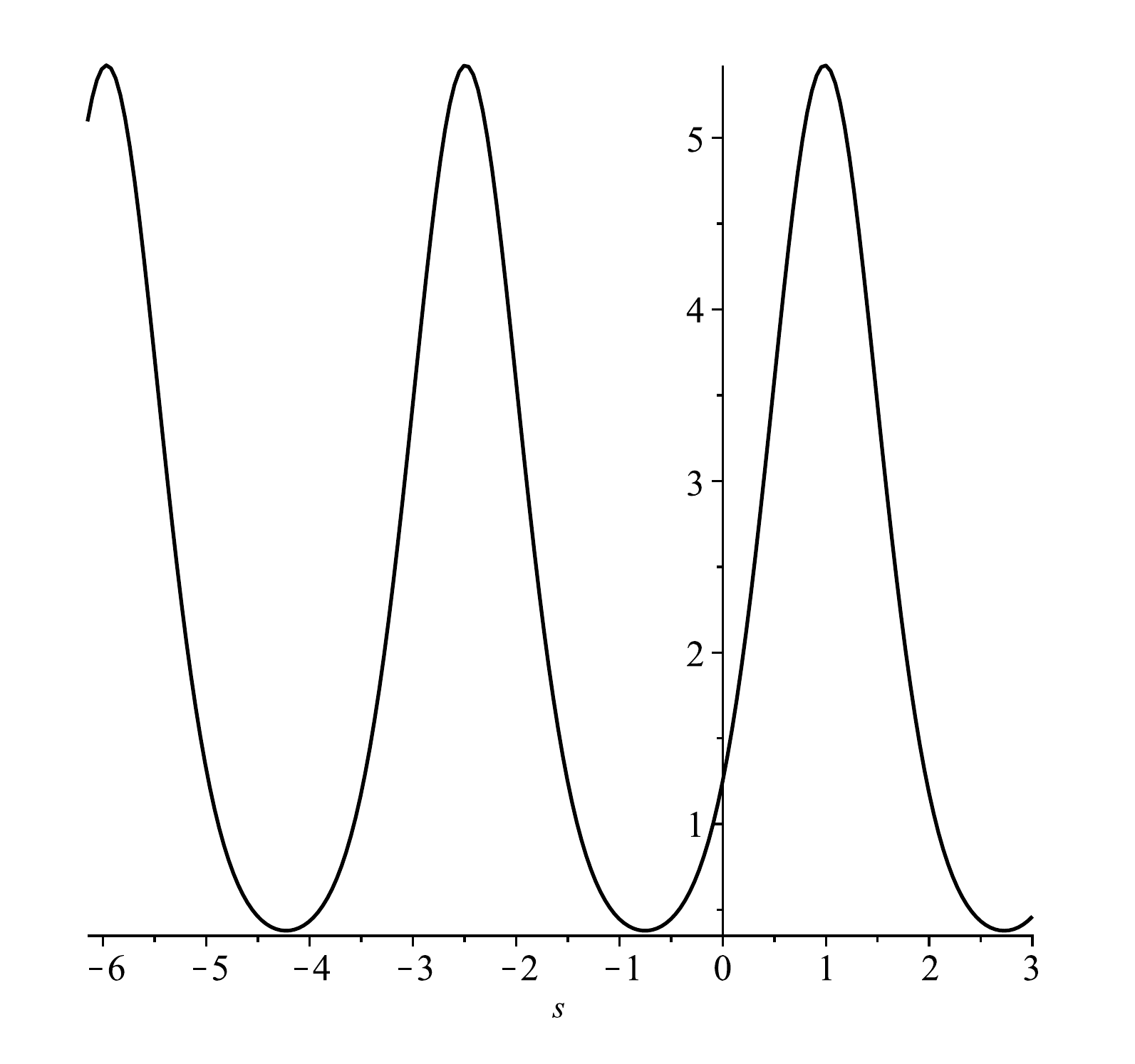}\\
     $\kappa_1$ vs $\kappa_2$ & $s$ vs $\theta$ & $s$ vs $\kappa$ \\
\end{tabular}

\begin{center}
\caption{Plots of the conservation laws\label{App1b}}
\begin{tabular}{cc}

    \includegraphics[width=0.3\textwidth]{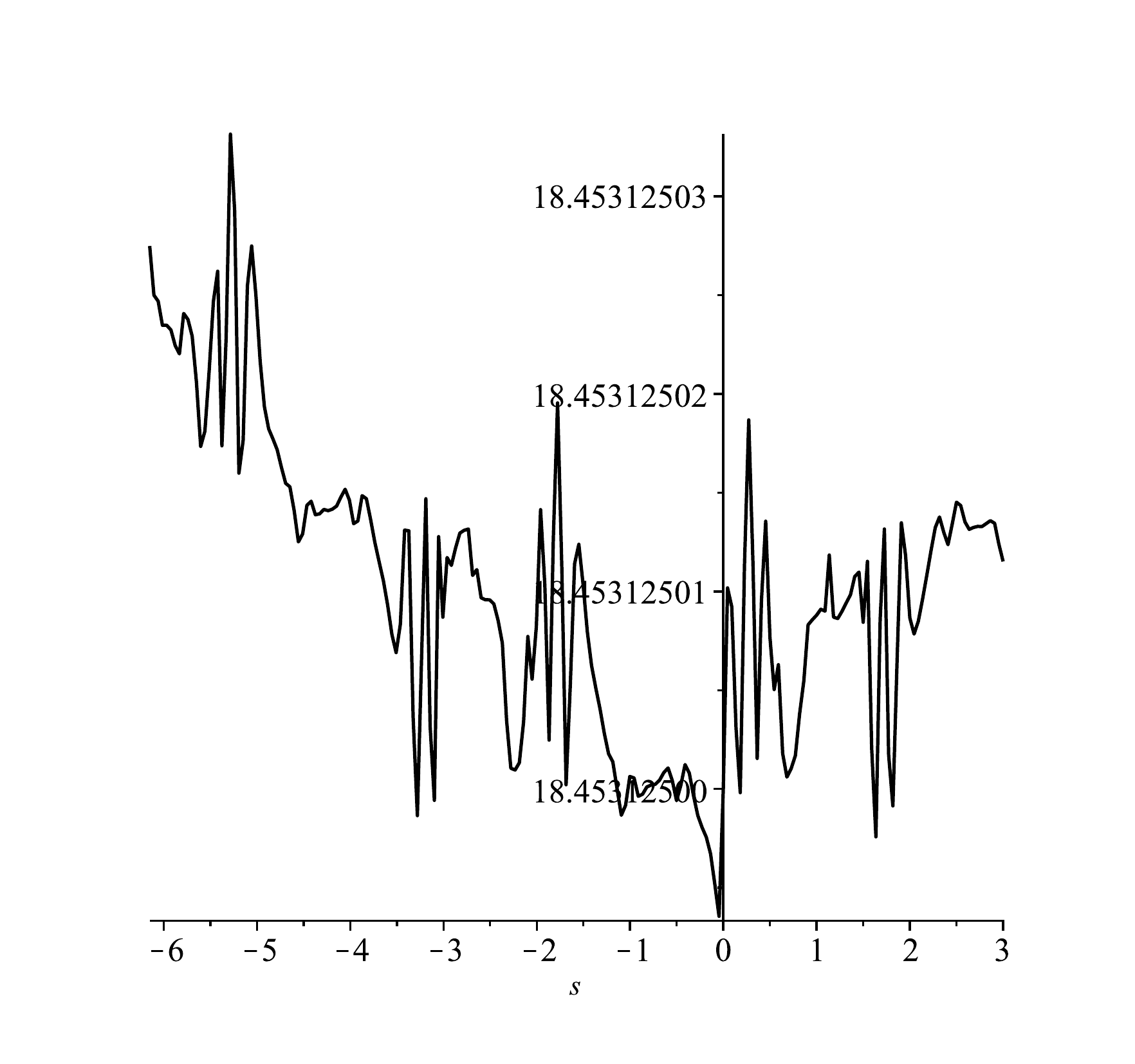} &

    \includegraphics[width=0.3\textwidth]{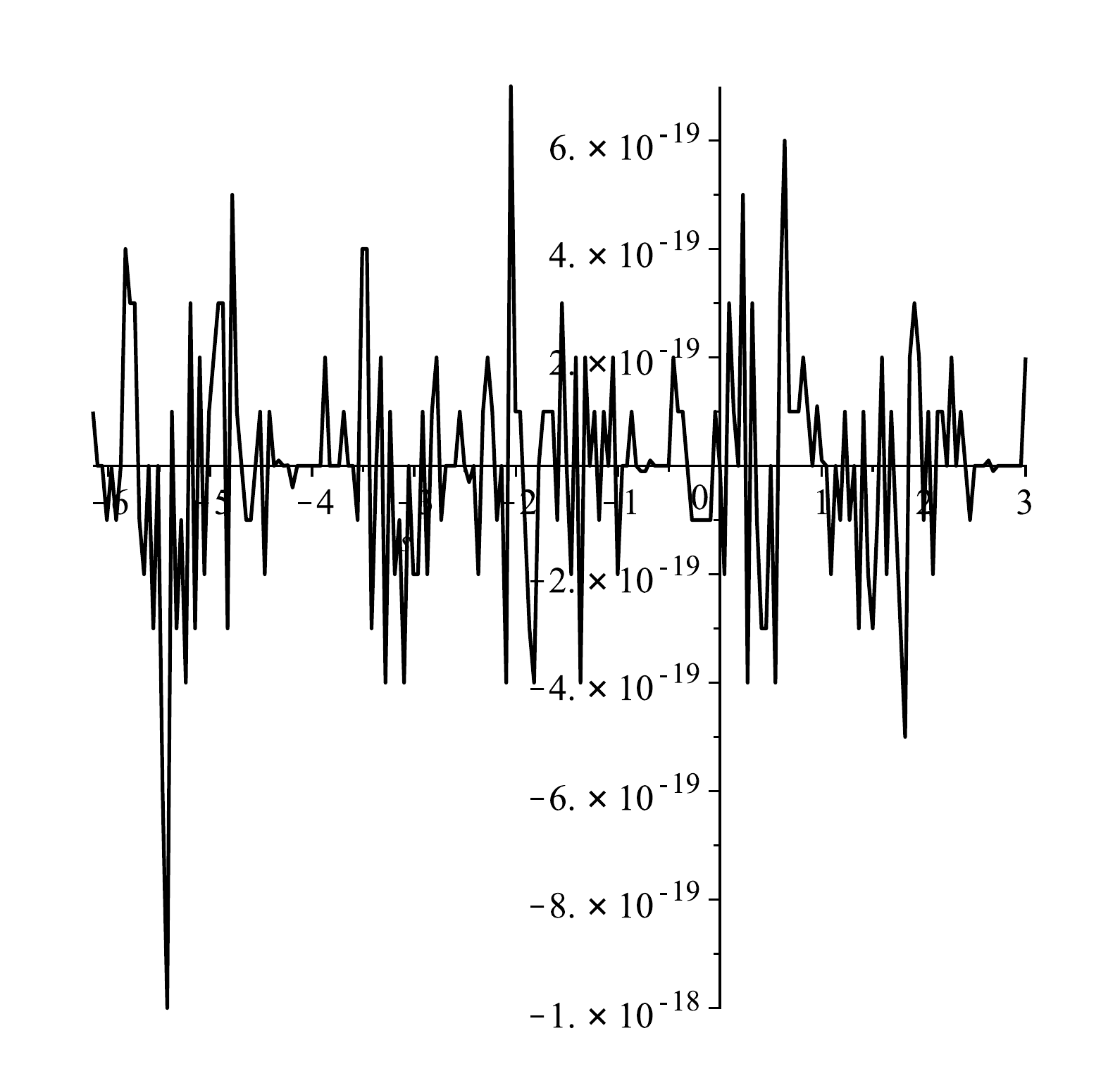} \\
     $v_1^2+v_2^2+v_3^2=c_1^2+c_2^2+c_3^3$  & $v_1v_4-v_2v_5+v_3v_6=c_1c_4-c_2c_5+c_3c_6$ \\
\end{tabular}
\end{center}

    \caption{Sweep surface using $V$ from the \RM frame along the extremal curve\label{App1c}}
    
    \vspace{5mm}
 \centering
    \includegraphics[width=0.6\textwidth]{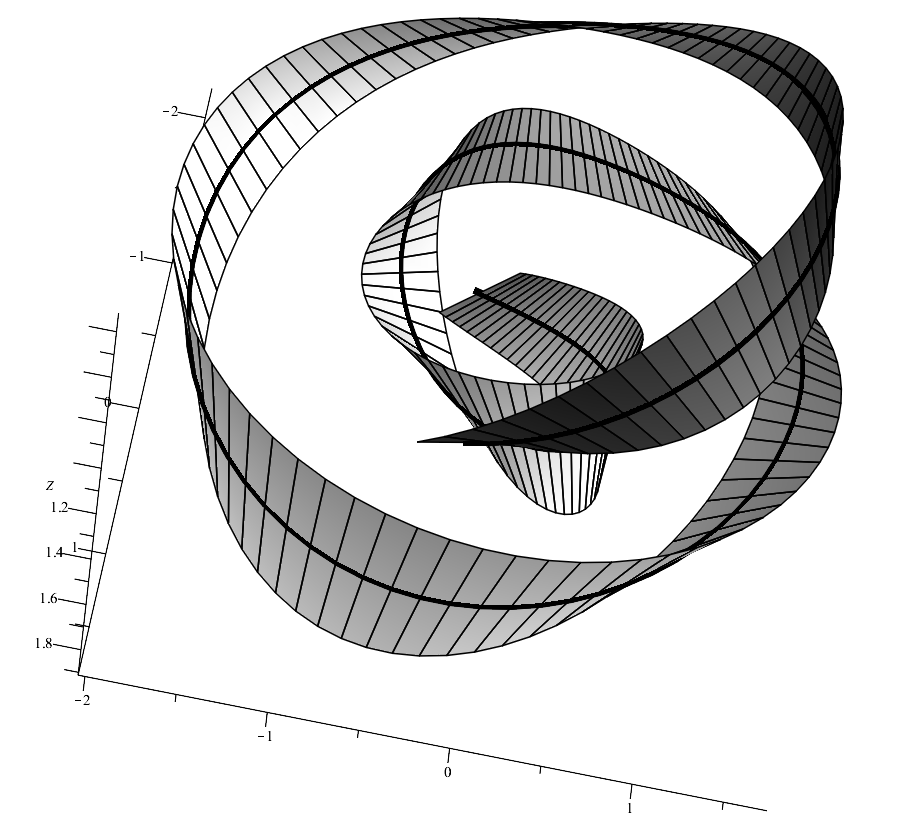}
\end{figure}

\subsubsection{The Lagrangian $\int \kappa_{1,s} \kappa_{2,ss} - \kappa_{1,ss}\kappa_{2,s}  \ {\rm d}s$}

We  now consider
\[
\int \kappa_{1,s} \kappa_{2,ss} - \kappa_{1,ss}\kappa_{2,s}  \ {\rm d}s
\]
In terms of the Euclidean curvature and torsion, this Lagrangian can be written as
$\int \left(\kappa^2\tau^3 + \tau\left(2\kappa_s^2-\kappa\kappa_{ss}\right)+\kappa\kappa_s\tau_s\right)\,{\rm d}s$.

The Euler--Lagrange equations are
\begin{align}\label{eulerlagrangelag3a}
-2\kappa_{2,ssss}+\frac{{\rm d}}{{\rm d} s}(\kappa_2 \mu)-\kappa_1\lambda&=0,\\\label{eulerlagrangelag3b}
2\kappa_{1,ssss}-\frac{{\rm d}}{{\rm d} s}(\kappa_1 \mu)-\kappa_2\lambda&=0
\end{align}
where
\[
\lambda =  2\kappa_{2,sss}\kappa_1 - 2\kappa_{1,s}\kappa_{2,ss} + 2\kappa_{2,s}\kappa_{1,ss}-2\kappa_2\kappa_{1,sss}
\]
and
\[
\mu=\kappa_{1,s}^2+\kappa_{2,s}^2-2(\kappa_1\kappa_{1,ss}+\kappa_2\kappa_{2,ss}).
\]

The conservation laws are of the form \eqref{conservatiolaws} where
\[
v(I)=(\lambda \quad  2\kappa_{2,ssss} - \mu\kappa_2  \quad  -2\kappa_{1,ssss} + \mu\kappa_1 \quad  \mu \quad 2\kappa_{1,sss}  \quad -2\kappa_{2,sss} ).
\]

Solving \eqref{eulerlagrangelag3a}, \eqref{eulerlagrangelag3b} along with  \eqref{PsiEqnCaseOne} and \eqref{PsiEqnCaseTwo} for $\kappa_1,\kappa_2$ and $\psi$ with initial conditions
\[
\begin{array}{lllll}
\kappa_1(0) = 1, &\kappa_2(0) = \frac{1}{2}, &\frac{{\rm d}}{{\rm d} s}\kappa_1(0) = 1,& \frac{{\rm d}}{{\rm d} s}\kappa_2(0) = 1,&\\[10pt]
\frac{{\rm d}^2}{{\rm d} s^2}\kappa_1(0) = 1,& \frac{{\rm d}^2}{{\rm d} s^2}\kappa_2(0) = 1, & \frac{{\rm d}^3}{{\rm d} s^3}\kappa_1(0) = 1,& \frac{{\rm d}^3}{{\rm d} s^3}\kappa_2(0) = 1, &\psi(0)=0 
\end{array}
\]
and integrating to obtain the extremizing curve and its \RM frame, we obtain the following solutions,
 see Figures \ref{App2a}, \ref{App2b}, \ref{App2c}.

\begin{figure}
\caption{Solutions for the invariants $\kappa_1$,$\kappa_2$,$\theta$ and $\kappa$\label{App2a}}
\begin{tabular}{ccc}

    \includegraphics[width=0.3\textwidth]{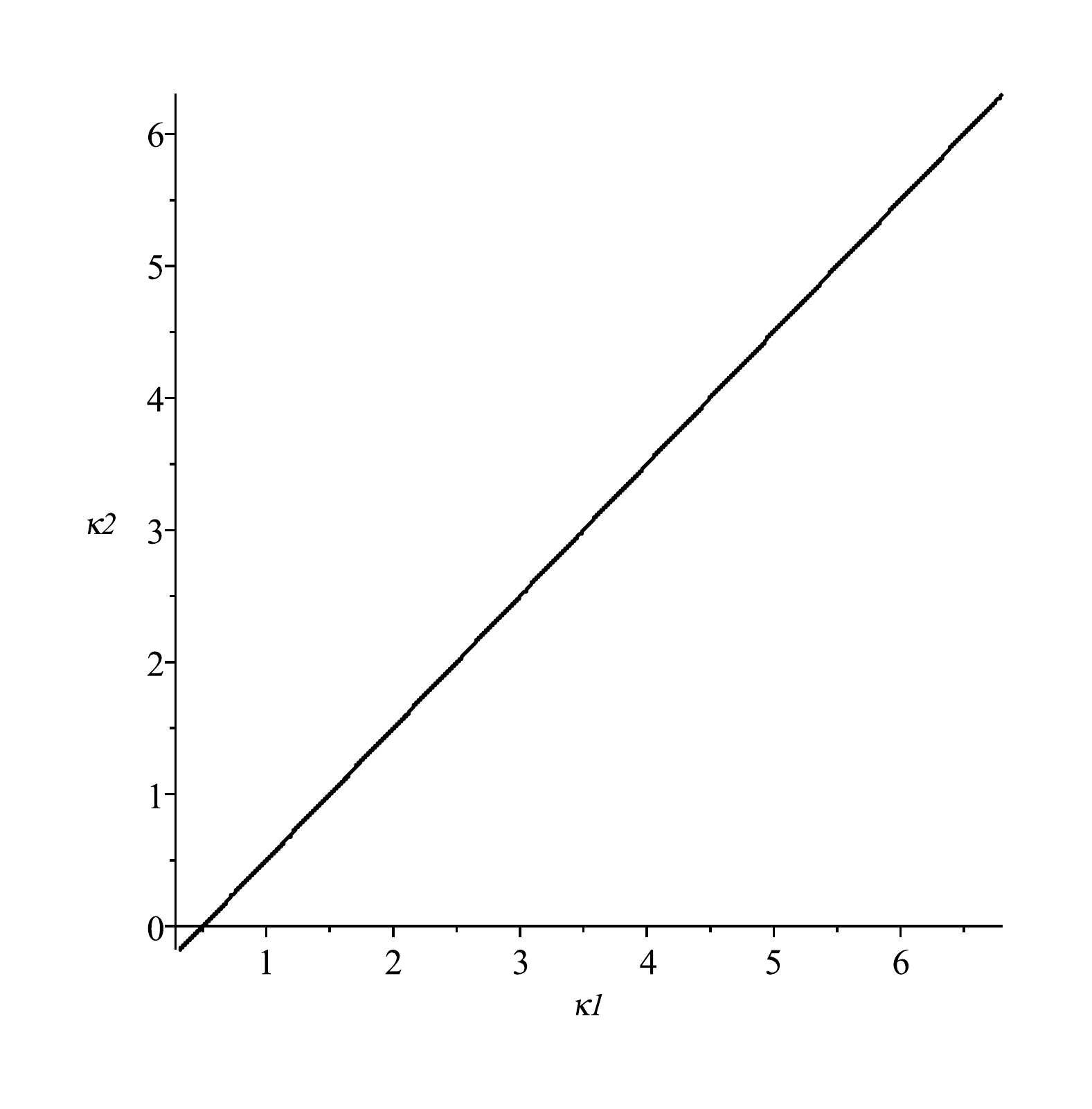} &

    \includegraphics[width=0.3\textwidth]{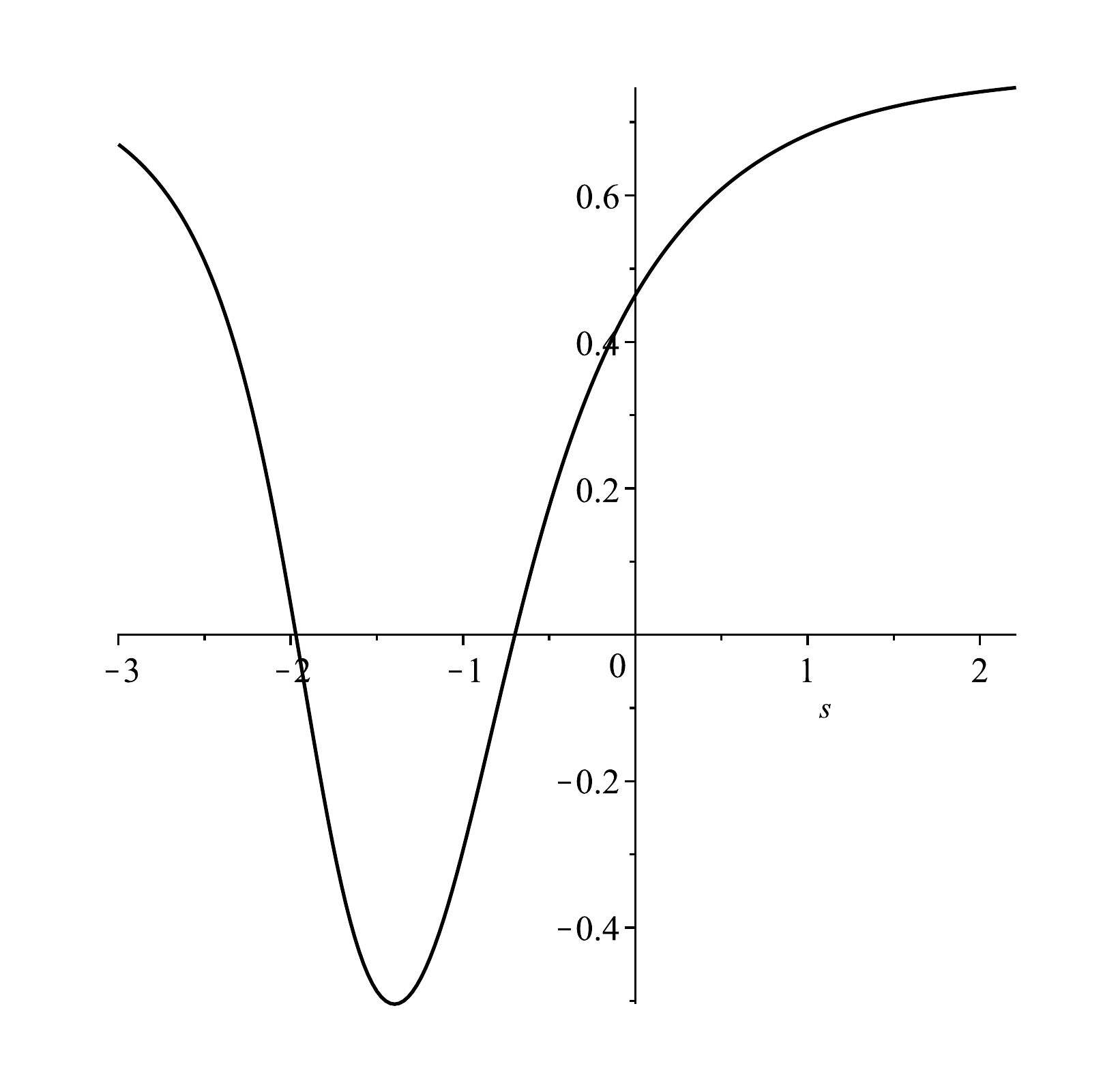} &

    \includegraphics[width=0.3\textwidth]{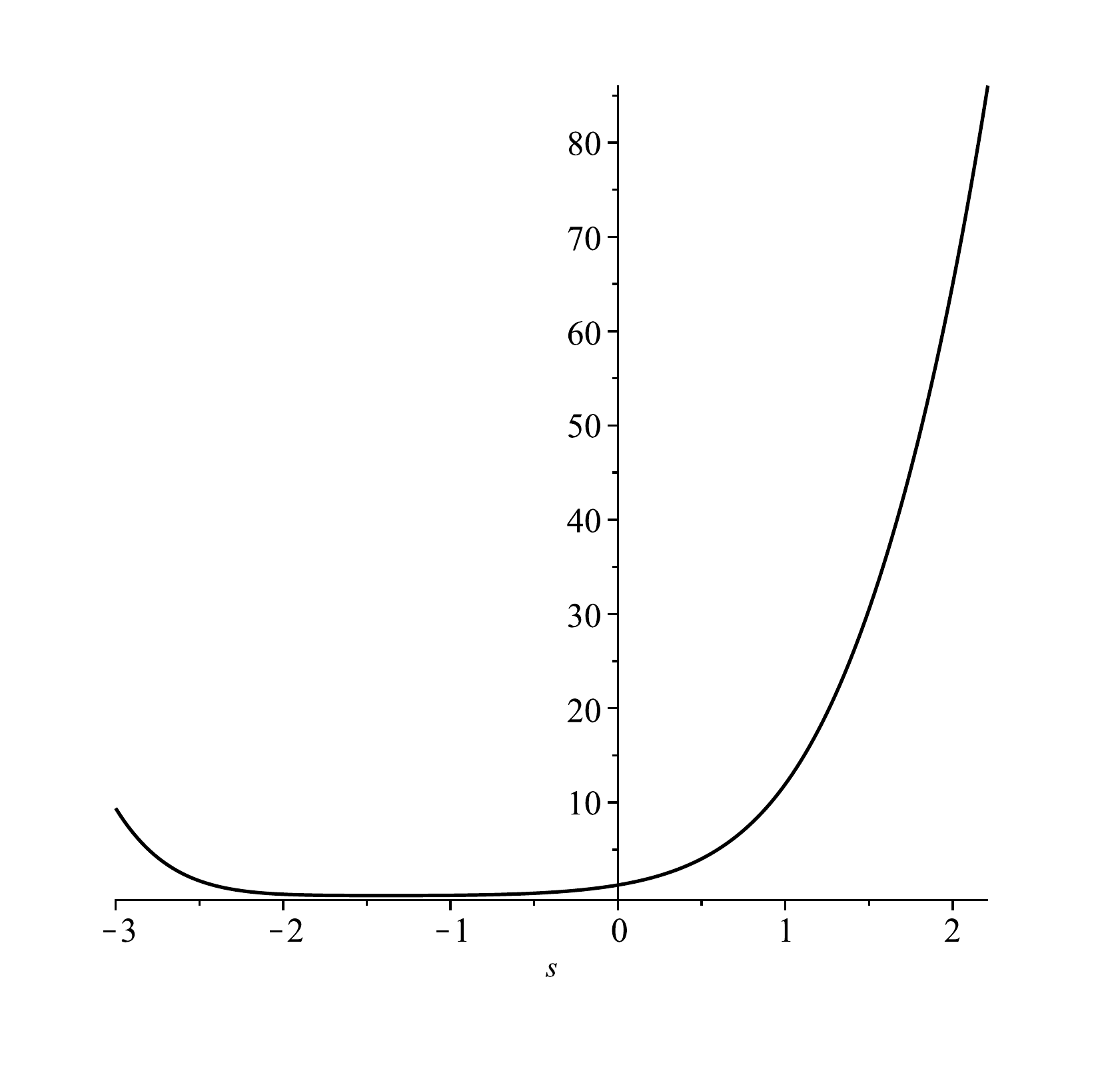}\\
     $\kappa_1$ vs $\kappa_2$ & $s$ vs $\theta$ & $s$ vs $\kappa$ \\
\end{tabular}

\begin{center}
\caption{Plots of the first integrals\label{App2b}}
\begin{tabular}{cc}

    \includegraphics[width=0.3\textwidth]{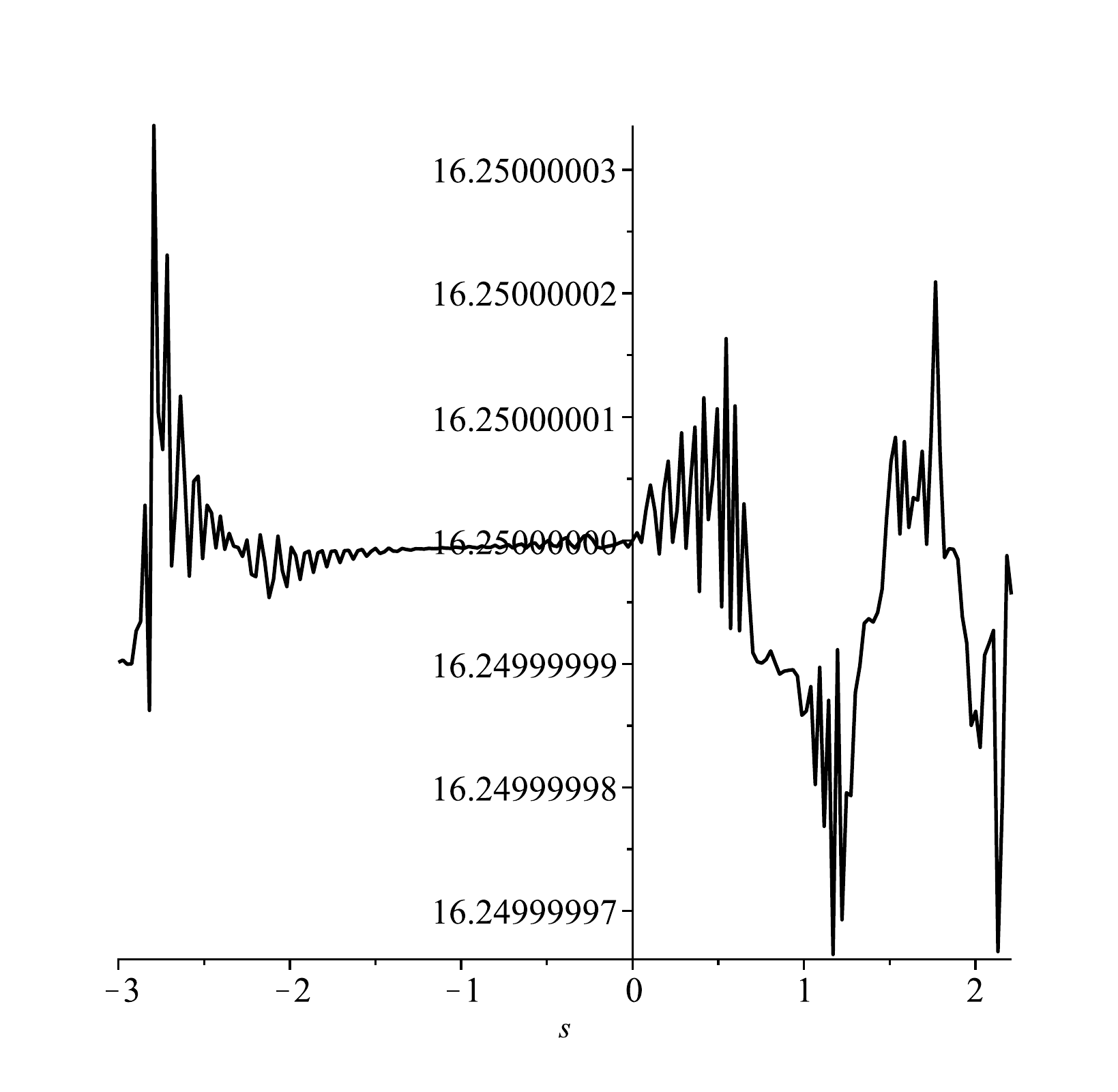} &

    \includegraphics[width=0.3\textwidth]{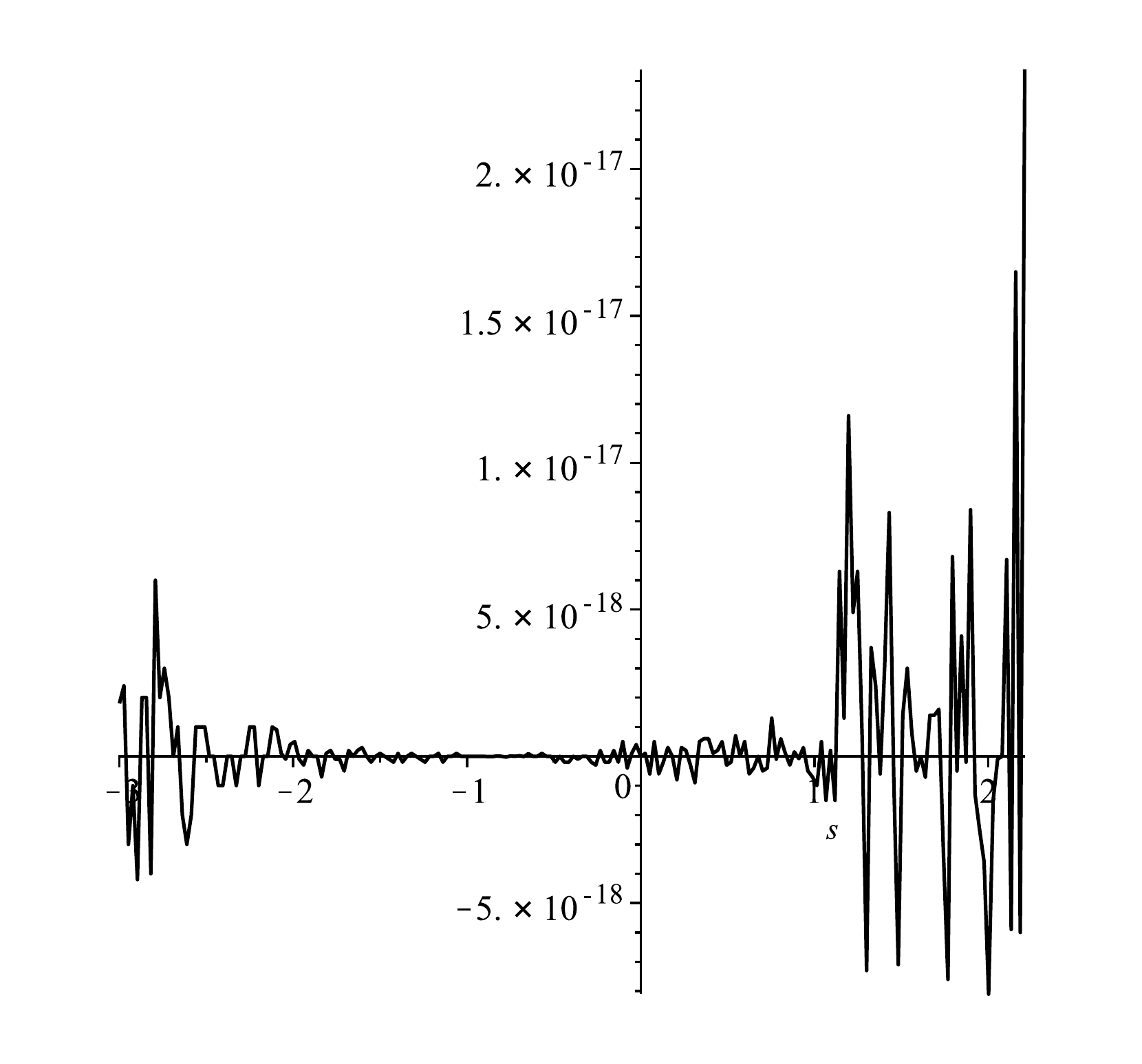} \\
     $v_1^2+v_2^2+v_3^2=c_1^2+c_2^2+c_3^3$  & $v_1v_4-v_2v_5+v_3v_6=c_1c_4-c_2c_5+c_3c_6$ \\
\end{tabular}
\end{center}

    \caption{Sweep surface using $V$ from the \RM frame along the extremal curve.\label{App2c}}
    
    \vspace{5mm}
 \centering
    \includegraphics[width=0.5\textwidth]{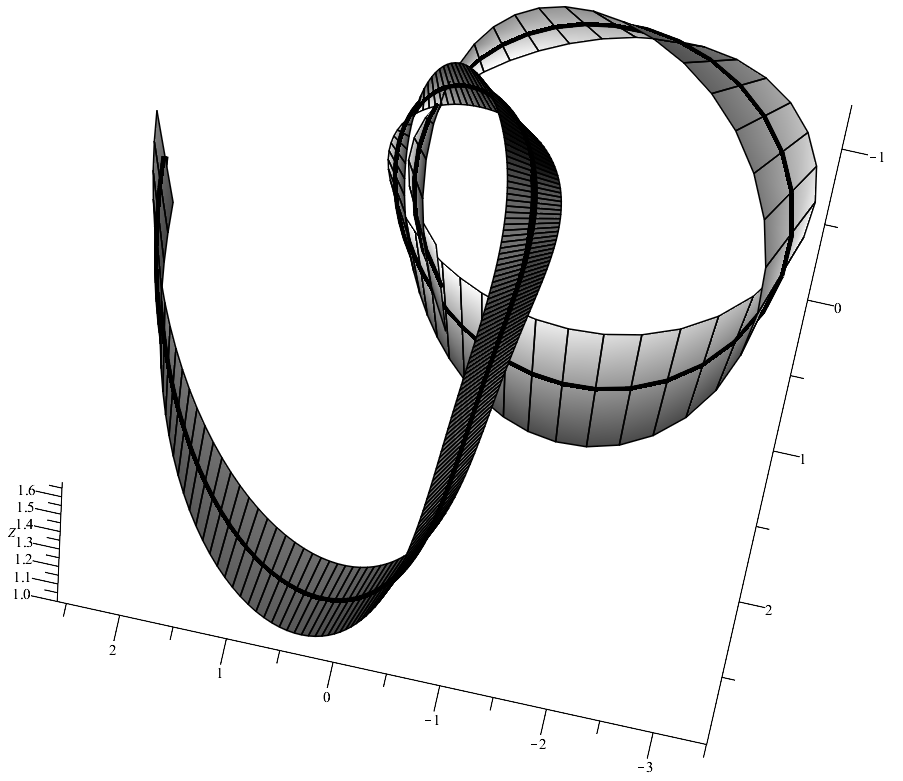}
\end{figure}

\section{Conclusions}\label{CCsection}

In this paper we have developed the Calculus of Variations for invariant Lagrangians under the Euclidean action of rotations and translations on curves in 3-space, using the \RM frame. We obtain the  Euler-Lagrange equations in their invariant form and their corresponding conservation laws. These results yield an easier form than those obtained in \cite{GonMan3}. We also show how to ease the integration problem using the conservation laws and to recover the extremals in the original variables. We show how to minimize the angle between the normal and binormal vector and give an application in the study of biological problems.

It is clear that our results can be generalized to obtain a symbolic calculus of invariants for a broad class of problems in which the frame is not
defined in terms of algebraic equations, in the coordinates of the manifold on which the Lie group actions. This is a topic for further study.

Future work would include the construction of a discrete \RM frame and obtaining the invariant Euler-Lagrange equations and conservation laws using the discrete invariant calculus of variations developed in \cite{newpaper}. The investigation of the minimization of functionals that are invariant under higher dimensional Euclidean actions is also of interest well as the study of joint invariants in problems where two helices appear and interact with each other. 

\section*{Acknowledgements}
The authors would like to thank Evelyne Hubert for pointing out the use of the \RM frame in the Computer Aided Design literature, and the Maplesoft Customer Support for help with the plots of the sweep surfaces, which were performed using Maple 2018. The second author 
 would like to thank the SMSAS department at the University of Kent and the EPSRC (grant EP/M506540/1) for generously funding this research.

\end{document}